\documentclass[a4paper, 11 pt, twoside]{amsart}
\usepackage[left = 3.5cm, right = 3.5cm, headsep = 6mm,
footskip = 10mm, top = 35mm, bottom = 35mm, footnotesep=5mm, headheight =
2cm]{geometry}
\usepackage{graphicx}
\usepackage[expansion=false]{microtype}
\usepackage{booktabs, multirow,  tabularx}

\usepackage{MnSymbol}
\usepackage[usenames, dvipsnames]{xcolor}

\usepackage{stmaryrd} 

\usepackage{tikz, tikz-cd}
\usepackage{enumitem}
\newlist{prooflist}{description}{1}
\setlist[prooflist]{font=\normalfont \itshape, labelindent = \parindent, leftmargin = 0pt}

\usepackage{verbatim}

\usepackage[unicode,bookmarks, pdftex, backref = page]{hyperref}
\hypersetup{colorlinks=true,citecolor=NavyBlue,linkcolor=NavyBlue,urlcolor=Orange, pdfpagemode=UseNone, breaklinks=true}

\usepackage{listings}

\lstdefinelanguage{Magma}{
  morekeywords={
    function,end,for,in,to,do,while,if,then,elif,else,return,
    true,false,repeat,until,break,continue,local,where
  },
  sensitive=true,
  morecomment=[l]{//},      
  morecomment=[s]{/*}{*/}, 
  morestring=[b]",          
}

\newtheoremstyle{special-example}
  {}
  {}
  {}
  {\parindent}
  {\bfseries}
  {:}
  { }
  {}

\renewcommand{\tilde}{\widetilde}

\newcommand{\pp}{\mathbb P}

\newcommand{\OO}{\mathcal O}

\newcommand{\wt}{\widetilde}

\newcommand{\inv}{^{-1}}

\newcommand{\I}{\mathrm{i}}

\newcommand\depth{\operatorname{depth}}

\newcommand\Supp{\operatorname{Supp}}

\newcommand{\mat}[1]{\begin{pmatrix} #1\end{pmatrix}}

\DeclareMathOperator{\Sym}{\mathsf{Sym}}

\newcommand{\refl}[1]{^{[{#1}]}}

\newtheoremstyle{Lehn-it}
  {}
  {}
  {\itshape}
  {}
  {\bfseries}
  {$\;$\textmd{---}}
  { }
  {}

\newtheoremstyle{Lehn-up}
  {}
  {}
  {\upshape}
  {}
  {\bfseries}
  {$\;$\textmd{---}}
  { }
  {}

  \newtheoremstyle{up-list}
  {}
  {}
  {\upshape}
  {}
  {\bfseries}
  { }
  { }
  {}

\newtheoremstyle{Lehn-Bemerkung}
  {}
  {}
  {}
  {}
  {\itshape}
  {$\;$\textmd{---}}
  { }
  {}

\newtheoremstyle{citing}
  {}
  {}
  {\itshape}
  {}
  {\bfseries}
  {$\;$\textmd{---}}
  {.5em}
  {\thmnote{#3}}
  
\numberwithin{equation}{section}

{\theoremstyle{Lehn-it}

\newtheorem{thm}[equation]{Theorem}
\newtheorem{lem}[equation]{Lemma}
\newtheorem{prop}[equation]{Proposition}
\newtheorem{cor}[equation]{Corollary}

}
{\theoremstyle{Lehn-up}
\newtheorem{defin}[equation]{Definition}
\newtheorem{exam}[equation]{Example}

}
{\theoremstyle{Lehn-Bemerkung}
\newtheorem{rem}[equation]{Remark}
}

{\theoremstyle{up-list}

}

{\theoremstyle{citing}
\newtheorem*{custom}{}}

\newcommand{\onto}{\twoheadrightarrow}
\newcommand{\into}{\hookrightarrow}

\DeclareFontFamily{OT1}{rsfs}{}
\DeclareFontShape{OT1}{rsfs}{n}{it}{<-> rsfs10}{}
\DeclareMathAlphabet{\curly}{OT1}{rsfs}{n}{it}

\DeclareMathOperator{\Proj}{\mathrm{Proj}}
\DeclareMathOperator{\Hom}{Hom}
\DeclareMathOperator{\Pic}{Pic}

\newcommand{\shom}{\curly{H}om}

\newcommand{\shext}{\curly{E}xt}

\DeclareMathOperator{\Aut}{Aut}
\DeclareMathOperator{\spec}{\rm Spec}

\DeclareMathOperator{\rk}{rk}

\DeclareMathOperator{\tensor}{\otimes}

\newcommand{\isom}{\cong}

\newcommand{\restr}[1]{{\raisebox{-0.1\height}{$|_{#1}$}}}

\renewcommand{\epsilon}{\varepsilon}
\renewcommand{\phi}{\varphi}
\renewcommand{\theta}{\vartheta}

\newcommand{\ke}{{\mathcal E}}
\newcommand{\kf}{{\mathcal F}}
\newcommand{\kg}{{\mathcal G}}
\newcommand{\kh}{{\mathcal H}}
\newcommand{\ki}{{\mathcal I}}

\newcommand{\kl}{{\mathcal L}}
\newcommand{\km}{{\mathcal M}}

\newcommand{\ko}{{\mathcal O}}

\newcommand{\IC}{{\mathbb C}}

\newcommand{\IN}{{\mathbb N}}

\newcommand{\IP}{{\mathbb P}}

\newcommand{\IZ}{{\mathbb Z}}

\title{Half canonical rings of Gorenstein spin curves of genus two}

\author[S. Coughlan]{Stephen Coughlan}
\address{Stephen Coughlan\\Current: Department of Mathematics and Computer Studies\\ Mary Immaculate College\\South Circular Road\\Limerick\\ Ireland}
\address{Previous: Institute of Mathematics of the Polish Academy of Sciences\\ul. \'Sniadeckich 8\\P.O. Box 21\\00-656 Warszawa\\Poland}
\email{stephen.coughlan@mic.ul.ie}

\author[M. Franciosi]{Marco Franciosi}
\address{Marco Franciosi\\Dipartimento di Matematica\\Universit\`a di Pisa \\Largo B. Pontecorvo 5\\I-56127  Pisa\\Italy}
\email{marco.franciosi@unipi.it}

\author[R.Pardini]{Rita Pardini}
\address{Rita Pardini\\Dipartimento di Matematica\\Universit\`a di Pisa \\Largo B. Pontecorvo 5\\I-56127  Pisa\\Italy}
\email{rita.pardini@unipi.it}

\author[S. Rollenske]{S\"onke Rollenske}
\address{S\"onke Rollenske\\FB 12/Mathematik und Informatik\\
Philipps-Universit\"at Marburg\\
Hans-Meerwein-Str. 6\\
35032 Marburg\\
Germany}
\email{rollenske@mathematik.uni-marburg.de}

\begin{document}

\begin{abstract}
We introduce the notion of generalised Gorenstein spin structure on a curve and we give an explicit description
of  the associated section ring for curves of genus two with ample canonical bundle, obtaining five different formats. 
\end{abstract}
\subjclass{14H42; 14H45, 14H20}
\keywords{spin structure, Gorenstein curve, section ring}
\maketitle

\setcounter{tocdepth}{1}
\tableofcontents

\section{Introduction}
Recall that a spin structure (or theta-characteristic)  on a smooth curve $C$ is a line bundle $\kl$ such that there is an isomorphism $\mu\colon\kl\otimes \kl \to \omega_C$. 
The interest in such structures originated in the classical study of theta functions in the 19th century and they were taken up again in 1970's in the context of the study of moduli spaces of curves 
(see \cite{Farkas-theta} for a recent survey). They also occur in the study of 
algebraic surfaces: if $X$ is a smooth projective surface and $C\in |K_X|$ is a 
smooth curve in the canonical linear system, 
then $\kl = K_X|_C$ is a theta-characteristic by the adjunction formula 
and the section ring $R(C, \kl)$ is often a stepping stone to 
compute the canonical ring of $X$ by Reid's hyperplane section principle.

In our study of the moduli space of stable surfaces, the need arose to consider 
a more general notion than what was available in the literature and which we 
hope to be of independent interest.

We define a generalised Gorenstein spin (ggs) structure on a Gorenstein curve  
$C$ as a  torsion-free sheaf $\kl$  of rank one  
at each generic point. 
 such that $\chi(\kl)=0$ together with  a \emph{multiplication map} $\mu\colon \kl\tensor \kl \to \omega_C$ 
which is an isomorphism at the generic points.

\begin{rem}\label{rem: motivation}
Let us sketch our motivation  for this definition.
Let $X$ be a stable surface   $K_X^2=1$ and $h^0(K_X)=2$ such that $2K_X$ is 
Cartier. Then any canonical curve $C\in |K_X|$ is a Gorenstein curve of 
arithmetic genus two  and with ample canonical bundle by adjuction. However, 
$\omega_X\tensor \ko_C$ might have torsion at the non-Gorenstein points of $X$. 
So we don't get a theta characteristic in the classical sense, but dividing out 
the torsion results in a ggs structure on $C$. 

The results of this paper are then an important puzzle piece in the complete 
classification of such surfaces $X$ in \cite{CFPR24}.
\end{rem}

 Inspired by Mumford's work 
\cite{mumford71}, Harris \cite{Ha82} considered the case of arbitrary 
curves with  $\kl$ a line bundle and the multiplication map being  an 
isomorphism. The Gorenstein condition is then implicit and we call this a 
\emph{regular ggs structure}. We will see in  Section \ref{sect: ggs} that 
in favourable situations,   the ggs structures on a curve $C$ are governed by 
regular ggs structures on partial normalisations. 

In later studies of the (compactified) moduli space of spin curves, initiated by 
Cornalba \cite{Cornalba89} and continued for example by Jarvis \cite{Jarvis98} 
and Ludwig \cite{Ludwig10}, spin structures are only considered on 
(semi-)stable curves. In fact, on such curves their  definition agrees with 
ours.\footnote{The condition on the cokernel of the multiplication map in 
\cite[Definition 2.1.2]{Jarvis98} is automatic by Proposition \ref{prop: 
irregular ggs}, Example \ref{example: A sing} and the degree condition.}

In our application to stable 
surfaces, Remark \ref{rem: motivation}, we need to  include  both curves with 
cusps and curves with elliptic tails in our analysis. As  explained in the 
comments  after Definition 1.2 in \cite{FedorchukSmyth2013}, these two classes 
of curves cannot occur at the same time in a separated moduli problem for 
curves, so it is not clear how to deal with  families in our setup. We have 
chosen not to pursue this question here.

Next, we  apply these results to reduced Gorenstein curves of genus two with ample canonical bundle, the class which originally motivated our research. 
 In Section \ref{section: description curves} we give a complete description of such curves,  showing that we can group them 
in   two types, Type A and Type B, distinguished by  the image of the bicanonical map (Proposition \ref{prop: types}) and moreover we classify all possible  ggs structures.
Such curves are hyperelliptic in the sense of \cite{BaBo24}, i.e. they admit a $2:1$ morphism onto a rational (not necessarily irreducible) 
curve (see \cite{Ho99} for the irreducible case).

We then go on and analyse the half-canonical ring of  ggs structures on such curves, that is, 
\[ R(C, \{ \kl, \omega_C\}) = \bigoplus_n R_n\] 
where $R_{2n} = H^0(C, nK_C) $ and $R_{2n+1} = H^0(C, \kl(nK_C))$. Note that this is not a standard section ring, because the multiplication of two elements of odd degree has to 
be defined via the multiplication map $\mu$.

Leaving the precise statements to Theorem \ref{thm: all half canonical rings}  we prove the following: 
\begin{custom}[Theorem]
 Let $( \kl, \mu)$ be a ggs structure on a reduced Gorenstein curve of genus two with ample canonical bundle. Then $h^0(\kl)= 0, 1, 2$ 
 and  we have a  complete description of the the half-canonical ring $R(C, \{\kl, \omega_C\})$  in terms of generators and relations
  according to the type of the curve and to the dimension  $h^0(\kl)$: 
 \begin{center}
  \begin{tabular}{cccc}\toprule
   $h^0(\kl)$ &0 &  1 & 2 \\
   Type of ring & $A(0)$, $B(0)$ &$A(1)$, $B(1)$ & $B(2)$ \\
   \bottomrule
  \end{tabular}
  \end{center}
In cases $A(0)$, $A(1)$, and $B(2)$ we get complete intersections, in cases $B(0)$ and $B(1)$ the rings admit a more complicated presentation in codimension six and four
(see  Table \ref{tab: half-canonical rings} for a complete description).
\end{custom}
Note that in each case we can interpret the ring as giving a natural embedding of $C$ into a weighted projective space $\IP$, such that $\omega_C = \ko_\IP(2)|_C$ and $\kl = \ko_\IP(1)|_C/\text{torsion}$. 

\subsection*{Acknowledgements}
We thank Igor Burban for some help with torsion-free sheaves on curve singularities. S.R.\ is grateful for support by the DFG. 
 M.F.  and R.P. are  partially supported by the project PRIN 
 2022BTA242 ``Geometry of algebraic structures: Moduli, Invariants, Deformations''
  of Italian MUR  and members of GNSAGA of INDAM.  
  
   We thank the anonymous referee for some improvements of the presentation and interesting remarks.

\section{Notation and preliminary results}\label{notation}

\subsection{Set-up}\label{ssec: setup}

We work with schemes of finite type over the complex numbers.

Given a sheaf $\kf$ on a scheme $X$, we denote by $\kf^{\vee}= \shom(\kf, \OO_X)$ its dual and 
 by $\kf^{[m]}$ the $m$-th reflexive power, i.e., the double dual of 
$\kf^{\tensor m}$. The sheaf  $\kf$ is reflexive if $\kf \cong \kf^{\vee\vee}$.

A variety is a reduced (possibly reducible)  Cohen--Macaulay  scheme.  
 
\subsubsection{Curves and partial normalisation} A curve $C$  is a Cohen--Macaulay   projective scheme of pure dimension 1.  It has a  dualising sheaf  $\omega_C$ and we denote by  $p_a(C)$ the arithmetic genus of $C$, that is, $p_a(C)=1-\chi(\OO_C)$. 
 The curve $C$ is Gorenstein if  $\omega_C$ is invertible; if this is the case  $K_C$ denotes a canonical divisor such that $\OO_C(K_C)\cong \omega_C$.
 
 We say that $p \in C$  is a singular point of type $A_{m}$ if 
 the completion of the local ring  $\OO_{C,p}$ is isomorphic to the ring  $\IC\llbracket x,y\rrbracket /(y^2 - x^{m+1}) $.

 Let $C$ be a reduced curve and $f\colon \tilde C \to C$ a finite birational morphism. Then the normalisation of $C$ factors over $f$, which is therefore often called a partial normalisation. 
	
 Note that a partial normalisation might not preserve the Gorenstein property in general: three  general  lines through the origin in $\IC^3$ are not Gorenstein, but map birationally to three general lines through the origin  in the plane, a Gorenstein singularity.
 
 For $A_m$ singularities these kinds of difficulties do not occur.
   
 \begin{exam}\label{example: A sing}
 	Let us describe all partial normalisations locally at a point of type $A_m$ where we have two cases according to the parity of $m$, compare  \cite[9.9, 5.11]{Yoshino}. 
 	\begin{description}
 		\item[$A_{2n}$]
 		The completion of the local ring is isomorphic to the ring  $\IC\llbracket x,y\rrbracket /(y^2 - x^{2n+1})\isom \IC\llbracket t^2, t^{2n+1}\rrbracket $ and the partial normalisations are given by the inclusions 
 		\[ \IC\llbracket t\rrbracket  \supset \IC\llbracket t^2, t^{3}\rrbracket \supset \IC\llbracket t^2, t^{5}\rrbracket \supset \IC\llbracket t^2, t^{7}\rrbracket \supset \dots\]
 		\item[$A_{2n-1}$]
 		The completion of the local ring is isomorphic to $\IC\llbracket x,y\rrbracket /(y^2 - x^{2n})$ and we get a series of inclusions
 		\[ \IC\llbracket x,y\rrbracket /(y-x)\times\IC\llbracket x,y\rrbracket /(y+x)\hookleftarrow \IC\llbracket x,y\rrbracket /(y^2 - x^{2})\supset \IC\llbracket x,xy\rrbracket \supset \IC\llbracket x,x^2y\rrbracket \supset \dots\]
 		where we consider $\IC\llbracket x,x^ky\rrbracket$  as a subring of 
 		$ \IC\llbracket x,y\rrbracket /(y^2 - x^{2})$, so that with $y_k = x^ky$ we get the desired relation $y_k ^2 - x^{2k+2} = 0 $.
 		
 	\end{description}
 	Note that in both series each successive subring has codimension one in the previous ring, when considered as a $\IC$-vector space.
 	
 	From this description we see that any partial normalisation of an $A_n$ singularity is an $A_m$ singularity with $m<n$ of the same parity, or smooth. In particular, it is also Gorenstein.
 \end{exam}

 \subsubsection{Group action} Assume that we have a finite group $G$ acting on a projective variety $X$ and  let
  $\kl$ be a rank one torsion-free  sheaf on $X$.  
  
  A linearisation of $G$ is an action of $G$ on $\kl$ that lifts the action of $G$ on $X$. If $G$ is cyclic, generated by $g$, 
   then the existence of a linearisation of $G$ is equivalent to say that there exists an isomorphism  $g^*\kl \cong \kl$. 

 A linearisation of $\kl$ induces a representation of $G$ on the vector space $H^0(X,\kl)$.  Conversely,  if $\kl$ is ample, one can recover the action of $G$ by considering 
 its action on $R(X,\kl)= \bigoplus_{d \in \IN}  H^0(X, \kl^{\otimes d})$. 
 We will apply these notions to  projective varieties  having an involution, which induces a decomposition of the space of sections in invariant and anti-invariants subspaces.

 \subsection{Weighted projective spaces}
We recall some facts about weighted projective spaces, see \cite{Dol,Fletcher} for details.

 Let $x_0,\dots,x_n$ be coordinates on $\IC^{n+1}$ and let $(w_0, \dots, w_n)$ be positive integers with $\gcd\{w_1,\dots,\widehat{w}_i,\dots,w_n\}=1$ for all $i=1,\dots,n$.
  Then the (well-formed) weighted projective space $\IP(w_0, \dots, w_n)$ is defined as  $\IC^{n+1}/  \sim $, where the equivalence relation $\sim$ is given,  for $\lambda \in \IC^{\ast}$,  by
  $$(x_0, \dots, x_n) \sim (\lambda^{w_0}x_0, \dots, \lambda^{w_n}x_n). $$ 
  $\IP(w_0, \dots, w_n)$ can be seen as $\Proj(S)$ where $S=\IC[x_0, \dots, x_n]$ with $\deg(x_i)=w_i$. 
  We  can write  $S= \bigoplus_{d \in \IN} S_d$, where $S_d \subset S$ is
   the additive subgroup consisting of homogeneous polynomials of degree $d$. For every $k  \in \IZ$ we have an induced  coherent sheaf $\ko(k)$ defined as 
  $ \tilde{S(k)}$, where $S(k)=\bigoplus_{d\in\IN}S(k)_d$ is the shifted $S$-module with $S(k)_d:=S_{k+d}$, where $S_i=0$ if $i<0$ by convention. 
  Weighted projective spaces have a natural structure of orbifolds and have singular points where the $\IC^\ast$-action has non-trivial stabilizer. The sheaf 
  $\ko(d)$ is invertible iff $d$ divides the least common multiple of $(w_0,\dots, w_n)$. The canonical sheaf of $\IP(w_0,\dots,w_n)$ is $\omega_{\IP}=\ko(-\sum_i w_i)$.

Suppose that $R=\bigoplus_{d\ge0}R_d$ is a finitely generated graded ring. 
Then $X=\Proj(R)$ naturally embeds in a weighted projective space. More precisely, if $R$ is generated by $x_0,\dots, x_n$ with $x_i$ in $R_{w_i}$, 
then the surjection $S=\IC[x_0,\dots,x_n]\to R$ induces an embedding of $X$ in $\IP(w_0, \dots, w_n)$ with structure sheaf $\ko_X=\widetilde{R}$. 
The equations defining the image of $X$ are given by the weighted homogeneous relations in $R$  among the generators (the kernel of the above surjection). 
One can compute the free resolution of $\ko_X$ as an $\ko_{\IP}$-module (equivalently, of $R$ as an $S$-module) exactly as in the classical case to get
\[0\to\kf_c\to\dots\to\kf_0\to \ko_X\to0\]
where $\kf_i$ are free $\ko_{\IP}$-modules, and $\kf_0=\ko_{\IP}$. If $c=n-\dim X$ then $\ko_X$ is Cohen--Macaulay by the Auslander--Buchsbaum theorem
 and the canonical sheaf of $X$ is $\omega_X=\ke xt^c(\ko_X,\omega_{\IP})=\kh om(\kf_c,\omega_{\IP})$.

\subsection{ Rank one torsion-free sheaves on Gorenstein curves }
A torsion-free sheaf on a curve $C$ is a coherent sheaf $\kf$ such that $\depth (\kf_x )=1$  for every closed point $x\in C$;  
 equivalently  $\kf$ has no  subsheaves supported on a finite set of points, i.e.,  $\kf$ satisfies   Serre condition $S_1$.

 \hfill\break
 {\bf Convention.} By a rank one torsion free sheaf on a curve $C$ we mean a torsion free sheaf  that has rank one at the general point of any component of $C$.

 The  degree of  a   rank one torsion-free sheaf $\kf$ on a curve $C$   is defined by the formula  $\deg(\kf) = \chi (\kf) - \chi (\OO)$, so that the Riemann--Roch formula is true by definition.
 \begin{rem}\label{rem: degree stuff on C}
  Note that by Serre duality on a Cohen--Macaulay curve $C$  we still have $\chi(\omega_C) = -\chi(\ko_C)$, so $\deg \omega_C = 2p_a(C) -2$.
If $\kf \into \kf'$ is an injection with finite cokernel,  then $\deg\kf \leq \deg \kf'$ by the additivity of the Euler characteristic.

\end{rem}
 In general, every  reflexive sheaf on a curve is torsion-free and by a result of Hartshorne  these notions coincide for rank one sheaves on a Gorenstein curve.
  \begin{lem}[\cite{Har--gen}, Prop.\ 1.6] \label{lem:reflexive} 
 Let $C$ be a Gorenstein curve and $\kf$ a rank one coherent sheaf on $C$. Then $\kf$  is reflexive if and only if it is 
 torsion-free. 
 
 In particular, for any  such sheaf $\kf$ we have
 \[ \kf^{[m]} = \left(\kf^{\tensor m}\right)^{\vee\vee} = \left(\kf^{\tensor m}\right)/\langle\text{0-dimensional torsion}\rangle.\]
\end{lem}

 Moreover, in this case,  if $\kf$ is a rank one  torsion-free sheaf and $\kh$ is an invertible sheaf then $\deg (\kf\otimes\kh)=\deg(\kf)+\deg(\kh)$ (cf. \cite[Prop. 2.8]{Har--gen}).
 We point out a useful Lemma, which relates the behaviour of  a rank one torsion free  sheaf and its pull-back under a partial normalisation.

\begin{lem}\label{lem: pushforward} 
 Let $ \kf$  be a rank one torsion free sheaf   on a reduced Gorenstein curve  and assume that all the partial normalisations of $C$ are Gorenstein. 
 Then:
 \begin{enumerate}
 \item there are a (uniquely determined)  partial normalisation $f\colon\tilde C \to C$ and a line bundle   $\tilde \kf$  on $\tilde C$ such that 
 $\kf = f_* \tilde \kf$;
 \item $\Aut(\kf)=\Aut(\tilde \kf)$. In particular, if $\tilde C$ is connected then $\Aut(\kf)=\IC^*$.
 \end{enumerate}
 \end{lem}

\begin{proof}
$(i)$ Since at a smooth point of a curve every torsion-free sheaf is free, the question is local at the singular points of $C$. 
Then the existence of the partial normalisation and the line bundle $\tilde \kf$ follows from  \cite[Prop. 7.2]{Bass63}. 
For the unicity, note that the natural map $f^*\kf =f^*f_*\tilde\kf\to \tilde \kf$ is surjective and therefore it  induces an isomorphism $(f^*\kf)^{[1]}\cong \tilde \kf$,  since both sheaves have rank one and $\tilde\kf$ is locally free. Hence $\tilde\kf$ can be recovered from $\kf$.

$(ii)$ Automorphisms  of $\kf$, $\tilde \kf$ are in bijection.  In fact an automorphism  of $\tilde\kf$ obviously induces an automorphism  of $\kf:=f_*\tilde\kf$. Conversely an automorphism of $\kf$ induces an automorphism of $(f^*\kf)^{[1]}\cong \tilde \kf$.
\end{proof}

The result that follows is certainly well known to experts, but we give a proof for lack of a suitable reference.
\begin{lem}\label{lem:  sequence-for-omega}
Let $C$ be a reduced Gorenstein curve and let $A,B\subset C$ be subcurves without common components such that $C=A\cup B$. Then   there are a natural  isomorphism  $\ki_{A} \omega_C \cong \omega_{B}$ and  an exact sequence:
$$0\to \omega_B\to\omega_C\to\omega_C\restr{A}\to 0.$$

\end{lem}
\begin{proof}Let $\ki_A,\ki_B\subset \ko_C$ be the ideals of $A$, $B$, respectively. Both ideals are radical and they satisfy $\ki_A\ki_B=0$. More precisely, $\ki_A$ is the annihilator of $\ki_B$, and viceversa. 
By relative duality for finite morphisms \cite[\href{https://stacks.math.columbia.edu/tag/0AU3}{Section 0AU3}]{stacks-project}  we have an inclusion $\omega_B=\shom_{\OO_C}(\ko_B,\omega_C)\into \omega_C$,
 so it is enough to give an isomorphism $\ki_A\omega_C\to \shom_{\OO_C}(\ko_B,\omega_C)$. 
 In fact, there is a natural map $F\colon \ki_A\omega_C\to \shom_{\OO_C}(\ko_B,\omega_C)$, that maps $\sigma\in \ki_A\omega_C$  to the homomorphism that sends $1\in \OO_B$ to $\sigma$  and, 
  using the fact that $\omega_C$ is invertible and $\ki_A$ is the annihilator of $\ki_B$, it is immediate to check  that $F$ is an isomorphism. 
\end{proof}

For later reference, we recall two results related to sections and cohomology of torsion-free sheaves.
\begin{lem}\label{vanishingH^1}
 Let $C$ be a  reduced Gorenstein curve and $\kf$ a rank one  torsion-free sheaf on $C$. If $\deg\left(\kf\restr{B}\refl{1}\right) \geq 2p_a(B)-1$  for every  subcurve $B\subseteq C$, then $H^1(C,\kf)=0$. 
\end{lem}

\begin{proof} 
The proof follows by the arguments used in \cite[Thm.1.1]{CFHR}. 
By Serre duality $H^1(C,\kf)^{\vee} \cong \Hom (\kf, \omega_C)$. Assume it is not zero and 
pick any nonzero map
$\varphi \colon \kf \to \omega_C$. By \cite[Lemma 2.4]{CFHR}, $\varphi $ comes from a generically surjective 
map 
$\kf\restr{B}\to\omega_B$ for a subcurve $B\subseteq C$  and   yields  an injection $\kf\restr{B}\refl{1} \into \omega_B$  whose 
cokernel has
finite length. Therefore   $\chi\left(\kf\restr{B}\refl{1}\right)\leq\chi(\omega_B)$, and Remark \ref{rem: degree stuff on C} gives
$\deg\left(\kf\restr{B}\refl{1}\right)   \leq 2p_a(B)-2$, against the assumptions.  
\end{proof}

To conclude this section we recall a slight generalisation of the {\em Base-Point Free Pencil Trick},  a way to study the behavior of multiplication of sections on a connected curve
(see \cite[\S 1 Thm.~2]{mumford70},   \cite[Prop.~1.5]{Franciosi}).

 \begin{prop}\label{prop: surjectivity}
Let $C$ be a connected  reduced Gorenstein curve,  let $\kh$ be  an invertible sheaf  
on $C$ such that $|\kh|$ is a base point free system, and let $\kf$ be a rank one torsion-free   sheaf on $C$. 
 Then   
the multiplication map
$$\rho_{\kf,\kh}:  H^0(C,\kh)\otimes H^0(C,\kf) \rightarrow H^0(C,\kh\otimes \kf) $$ is surjective if one of the following conditions  holds

 \begin{enumerate} 	 
	 \item 
$H^1(C,  \kf\otimes \kh^{-1})=0$; by Lemma \ref{vanishingH^1} this holds 
 if $\deg \left( (\kf\otimes \kh^{-1})_{|B}\refl{1}\right) \geq  2 p_a(B) - 1$
 for every subcurve $B \subseteq C$.

\item  $h^0(C,\kh)\geq 3 $ and $\kf\otimes \kh^{-1}\cong \omega_C$. 
 \end{enumerate}

Moreover,  if $h^0(C,\kh)=2$  then 
 $\ker\rho_{\kf,\kh} \cong H^0(C, \kf\otimes \kh^{-1}).$

\end{prop}

\begin{proof}   
(i)  follows by adapting to our case  the arguments used in   \cite[\S 1 Thm.~2]{mumford70} (which is stated for any coherent sheaf $\kf$  on a variety $X$). Indeed,   
 since $|\kh|$ is a base point free system, we can  take a section which yields an inclusion  $\kf \otimes \kh^{-1} \to \kf$; denoting by 
 $\kg$ its cokernel, by our assumptions  $\dim (\Supp (\kg))=0 $  and we have an exact sequence
 $$ 0 \to H^0(C,  \kf\otimes \kh^{-1})\otimes H^0(C,\kh) \to H^0(C,  \kf) \otimes H^0(C,\kh) \to  H^0(C,  \kg)\otimes H^0(C,\kh) \to 0 ;$$
 by a direct computation we may assume  $H^0(C,\kh)\otimes H^0(C,\kg) \onto H^0(C,\kh\otimes \kg) $, and we may conclude 
 as in   \cite[\S 1 Thm.~2, p. 45]{mumford70} by diagram chasing. 
 
$(ii)$  follows from  \cite[Prop.~1.5]{Franciosi}, since in this case  $\kf$ is invertible. 
 
 The last statement is the classical Castelnuovo's Base-Point Free Pencil Trick as shown in \cite[p. 126]{ACGH}.
\end{proof}

 \section{Reduced Gorenstein curves of genus  two with ample canonical bundle }\label{section: description curves}
 The contribution of a singular point of a curve to the global genus, compared to the normalisation, is measured by the so-called $\delta$-invariant of the point. Singularities with small value of this invariant have been classified by Stevens \cite{Stevens96}  and later independently in \cite[Appendix A]{Smyth11} and \cite[Sect. 2]{Battistella22}. However, without further conditions, there are infinitely many Gorenstein curves of any fixed genus. 
 \begin{exam}
  If $C$ is any nodal curve such that every component is isomorphic to $\IP^1$ and the dual graph is a tree, then $C$ is Gorenstein of  arithmetic genus zero. 
  
    Clearly, every irreducible curve of genus zero is isomorphic to $ \IP^1$. 
   \end{exam}
\begin{exam}\label{exam: genus 1 Gorenstein curves}
 There are three different types of irreducible Gorenstein  curves $C$ of arithmetic genus $1$, 
  namely a smooth elliptic curve, or a rational curve with a node $A_1$ or a cusp $A_2$. 
 
 This can either be deduced from Stevens' classification or by showing that a line bundle of degree three embeds $C$ as a plane cubic reasoning as in Proposition \ref{prop: types} below.
\end{exam}

 There are several ways to understand and describe the curves we are interested in. For sake of brevity, we take a quick route via the canonical ring, where we double the degrees  in consistency with later considerations.
\begin{prop}\label{prop: types}
 Let $C$ be a reduced connected  Gorenstein curve of arithmetic genus two with ample canonical bundle. Then the canonical ring is isomorphic to one of the following:
 \begin{description}
  \item[Type A] $R(C, \omega_C) = \IC[y_1, y_2, u]/( u^2 -f_{12}(y_1, y_2)),$
 where $\deg(y_1, y_2, u) = (2,2,6)$ and $f_{12}$ is a non-zero polynomial of weighted degree $12$. 
  \item[Type B] $R(C, \omega_C) \isom \IC[y_1, y_2,v,  u]/ ( y_1y_2, u^2 -f_{12}(y_1, y_2, v)),$
 where $\deg(y_1, y_2, v, u) = (2,2,4, 6)$ and $f_{12}$ is a  polynomial of degree $12$ containing the monomial $v^3$ with non-zero coefficient. 
 \end{description}
 \end{prop}
\begin{proof}
Note that for a Gorenstein curve the Hilbert series of the canonical ring is determined by the genus and the Riemann--Roch formula,  so in our case the series is $H(t)=1+2t^2+\sum_{n\ge2}(2n-1)t^{2n}$. Since 
\[t^2H(t)=t^2+2t^4+\sum_{n\ge3}(2n-3)t^{2n},\] we may subtract the latter series from the former to get  
\[(1-t^2)H(t)=1+t^2+t^4+\sum_{n\ge3}2t^{2n}.\]
 Repeating the above procedure, we have 
 \[t^2(1-t^2)H(t)=t^2+t^4+t^6+\sum_{n\ge4}2t^{2n}\]
 and subtracting yields $(1-t^2)^2H(t)=1+t^6$. Hence 
 \[H(t)=\frac{1+t^6}{(1-t^2)^2}=\frac{ (1-t^{12})}{(1-t^2)^2(1-t^6)}.\]

If $C$ is irreducible, we can proceed as for smooth curves (compare \cite[Section~3.1]{FPR17a}) and get a canonical ring of type A.

Now assume $C$ is reducible.  Note that an ample line bundle has positive degree on any component, so there are at most two components and we can write $C = C_1\cup C_2$ with $\deg(\omega_C|_{C_i}) = 1$.
 
By Lemma  \ref{lem:  sequence-for-omega}  we have an isomorphism  $\ki_{C_{2}} \omega_C \cong \omega_{C_{1}}$ and   there is an exact sequence 
\begin{equation}\label{eq: canonical decomposition}
 0 \to \omega_{C_1} \to \omega_C \to \omega_C|_{C_2}\to 0.
\end{equation}
This implies that $\omega_{C_i} \into \omega_C|_{C_i}$, 
so by Remark \ref{rem: degree stuff on C} we get  $\deg \omega_{C_i} \leq 0$ by parity and thus $C_i$ has arithmetic genus at most one. Therefore, each component is Gorenstein by \cite[Lemma 1.19]{catanese82}.

Taking global sections in \eqref{eq: canonical decomposition} after twisting by $(m-1)K_C$ with  $m\ge 1$,  we get an exact sequence 
\begin{equation}\label{eq: restriction sequence}
 0\to H^0(K_{C_1}+(m-1)  K_C|_{C_1}) \to H^0(m K_C) \to H^0(m K_C|_{C_2})\to0
 \end{equation}
 because by Serre duality  $H^1(K_{C_1}+(m-1)  K_C|_{C_1}) \to  H^1(m K_C)$ is an isomorphism for $m =1$  and both are zero for $m\geq 2$.

Now assume $p_a(C_2)=0$, that is, $C_2$ is a smooth rational curve. Then $\omega_C|_{C_2} = \ko_{C_2}(1)$ has two sections, so  sequence \eqref{eq: restriction sequence} for $m=1$ shows that $h^0(K_{C_1}) = 0 $, and therefore both components are rational. Also, since $K_C|_{C_i}$ has degree one, it is base-point free on the components and thus globally base point free, because all sections lift to $C$. 
 By Proposition \ref{prop: surjectivity} $(i)$  the multiplication map 
\[H^0(C, K_C) \otimes H^0(C, nK_C) \to H^0(C, (n+1)K_C)\]
is onto for $n \geq 3$.
For $n=1$ one can check easily   that $\langle y_1^2, y_2^2, y_1y_2\rangle$ span $H^0(C,2K_C)$, where $y_1, y_2$ is a basis of $H^0(K_C)$. 
  For $n=2$  the kernel of the multiplication map  has dimension two (generated by the trivial relations $y_1 \tensor y_1y_2 - y_2\tensor y_1^2, y_2 \tensor y_1y_2 - y_1\tensor y_2^2$),
   hence we need a new generator $u \in H^0(3K_C)$. 
   
   Concerning the relations, we obtain 
   a non-trivial  relation  between monomials in our generators in $H^0(6K_C)$, which we can write in the given form by completing the square. Comparing the Hilbert series,   we see that we have already found the full ring of type A.
   
Now we turn to the case where both components are of arithmetic genus one. By \cite[Lemma 1.12]{catanese82} the ideal $\ko_{C_2}(-C_1)$ is invertible and defines a subscheme of length $\deg \omega_C|_{C_2} - \deg \omega_{C_2} = 1$, i.e.,  a single point $p\in C$, which by \cite[Prop. 1.10]{catanese82} is a node of $C$.
Then sequence \eqref{eq: restriction sequence} for $m = 1$ shows that 
$p$ is a base point of the canonical system. Let $y_i$ be a generator of the kernel of $H^0(K_C)\to H^0(K_C|_{C_i})$. Then $H^0(K_C) = \langle y_1, y_2\rangle$ such that $y_i |_{C_i} = 0 $.
In particular, $y_1y_2 = 0$ and we  can write $H^0(2K_C) = \langle y_1^2, y_2^2, v\rangle$. 
Note by \eqref{eq: restriction sequence}, this time for $m=2$, that $y_1^2, v$ restricts to $C_2$ as a complete linear system of degree two, thus without base points. Hence also $|2K_C|$ has no base points, $v$ is non-zero on both components, and 
   by  Proposition \ref{prop: surjectivity} $(ii)$ for $n=3$ and $(i)$ for $n\ge 4$,   we obtain  surjections 
 $$H^0(C, 2K_C) \otimes H^0(C, nK_C) \onto H^0(C, (n+2)K_C)$$  for  every $n \geq 3$.
  
 Now consider the multiplication map $H^0(K_C) \otimes H^0(2K_C) \to H^0(3K_C)$. The image of this map is generated by the monomials  $\{y_1^3,y_2^3, y_1v, y_2v\}$.
 We have no linear relations among them since a combination 
 \[a_1y_1^3 + a_2y_2^3 + b_1y_1v+b_2y_2v =0\]
 when restricted to $C_1$ yields $  a_2y_2^3 +b_2y_2v =0$, whence $a_2=b_2=0$,  and then restricted to $C_2$ gives $a_1=b_1=0$.
 Therefore, since $h^0(C,3K_C)=5$ we need a new generator $u$ of degree 6. 
 
  Similarly we can show that $H^0(4K_C)$ is generated by the monomials
 \[\{y_1^4,y_2^4, y_1^2v, y_2^2v, v^2, y_1 u , y_2u\}.\]
  Summing up, the canonical ring $R(C,K_C)$ is generated by $\{y_1,y_2,v,u\}$.

Concerning the relations, we have trivial relations  induced by $y_1y_2=0$ in degree $\leq 10$, and a new relation in  degree $12$, which after completing the square can be written as $u^2 -f_{12}(y_1, y_2, v)$. 
Note that $f_{12}$ must contain the monomial $v^3$ with
non-zero coefficient, since otherwise the intersection point $C_1 \cap C_2$ corresponding to
the ideal $(y_1, y_2, u^2 - f_{12}(y_1, y_2, v))$ would be the point $(0 : 0 : 1 : 0)$ and therefore singular for both $C_1$ and $C_2$.
  Comparing the Hilbert series, $\frac{ (1-t^{12})(1-t^4)}{(1-t^2)^2(1-t^6)(1-t^4)}=\frac{ (1-t^{12})}{(1-t^2)^2(1-t^6)}$
   we see that we have already found the full ring of type B.    \end{proof}

\begin{cor}\label{cor: bicanonical map}
  Let $C$ be a reduced connected  Gorenstein curve of arithmetic genus two with ample canonical bundle. 
 Then the following hold:
 \begin{enumerate}
  \item The bicanonical map is a morphism of degree two onto a conic $ Q $ in $\IP^2$.
  \item The conic $ Q$ is reducible if and only if $C$ is of  type B.
       \item The rational involution $\iota$  induced on $C$ by the bicanonical map is regular.  
  \end{enumerate}
\end{cor}
\begin{proof} $(i)$ and $(ii)$ follow immediately  by  Proposition \ref{prop: types}.

$(iii)$ By Proposition \ref{prop: types} the curve $C$ is embedded in a weighted  projective space $\pp$ in such a way that the regular  involution of $\pp$ that sends $u$ to $-u$  and leaves the remaining variables invariant induces $\iota$ on $C$. 
\end{proof}

\begin{rem}\label{rem: linearisation} Let $X$ be a 
projective variety, in particular reduced and  Cohen--Macaulay by our conventions,    with an involution $\iota$ that acts 
non-trivially on every irreducible component of $X$ and let $\pi\colon X\to 
Y:=X/\iota$ be the quotient map. Then  $\pi_*\ko_X$ decomposes under the action 
of $\iota$ as a direct sum
  $\ko_Y\oplus \kf$, where $\kf$ is a torsion-free sheaf of rank one, and $\iota$ acts on $\ko_Y$ as the identity and on $\kf$ as multiplication  by $-1$. By relative duality we have $\pi_*\omega_X=\shom(\pi_*\ko_X,\omega_Y)$
    and the decomposition of $\pi_*\ko_X$ induces a decomposition $\pi_*\omega_X = \omega_Y\oplus \shom(\kf,\omega_Y)$, where the first summand is $\iota$-invariant and the second one is anti-invariant.
     We call the corresponding linearisation  of $\omega_X$ 
\emph{natural}. Note that when both $X$ and $Y$ are smooth the 
natural linearisation  corresponds to the natural action on differential forms, 
for which invariant
      $n$-forms on $X$ correspond to pull-backs of $n$-forms on $Y$;  note also that if  the map $\pi$ is flat, namely $\kf$ is a line bundle,  then 
       $\omega_X=\varphi^*(\omega_Y\otimes \kf\inv)$ 
      and  the opposite of the natural linearisation is the one induced by the fact that $\omega_X$ is a pull-back from $Y$.  

Consider now  a reduced Gorenstein curve $C$ of genus two with $\omega_C$ ample; the bicanonical map $\varphi\colon C\to Q$ is the quotient map for the bicanonical involution $\iota$ defined in Corollary \ref{cor: bicanonical map}.  Actually in Corollary \ref{cor: bicanonical map} $\iota$ is defined by its action on the canonical ring $R(C,\omega_C)$, which in addition determines  a  linearisation of $\omega_C$ for which the elements  of  $H^0(\omega_C)$ are all invariant. Since $H^0(\omega_Q)=0$, this is  the opposite of the natural linearisation and is the linearisation that we will always consider in our computations 
\end{rem}

\begin{cor}\label{cor: linearisation} 
With  respect to the linearisation fixed in Remark  \ref{rem: linearisation}, the invariant subring $R(C, \omega_C)^+\subset R(C, \omega_C)$ has Hilbert series $(1-t^2)^{-2}$.
\end{cor}
\begin{proof}
The invariant subring $R(C, \omega_C)^+\subset R(C, \omega_C)$ is the ring generated by $y_1, y_2$ for type A and by $y_1, y_2, v$ for type B, so its  Hilbert series is $\frac{1}{(1-t^2)^{2}}=\frac{1-t^4}{(1-t^2)^2(1-t^4)}$. \end{proof}

We close this section by describing the geometry of the curves described in Proposition \ref{prop: types}.  The possibilities are illustrated in Table \ref{tab: list of curves}.

\begin{cor}\label{cor: description Types}
  Let $C$ be a reduced connected  Gorenstein curve of arithmetic genus two with ample canonical bundle.  Then 
 \begin{enumerate}
  \item If $C$ is of type A then:
  \begin{itemize}
  \item[(a)] the singularities of $C$ are  of type $A_{n_1},\dots A_{n_k}$, with $\sum_i (n_i+1)\le 6$;
  \item[(b)] $C$ is reducible iff $\sum_i (n_i+1)=6$ and all the $n_i$ are odd. 
In this case,  the components of $C$ are smooth rational curves exchanged by 
$\iota$.
  \end{itemize}
   \item If $C$ is of type $B$, then it is the union of two irreducible curves $C_i$  with $p_a(C_i)=1$ for $i=1,2$, that meet transversely at a single  point $p$ that is smooth for both.
  The involution  $\iota$ restricts on each component to an involution with smooth rational quotient curve. 
 \end{enumerate}
\end{cor}
 \begin{proof} $(i)$ Up to a linear change in the coordinates $y_1,y_2$ we may assume that an open set  $U$ of $C$ containing all the singularities is $ U=\{ u^2-\Pi_1^r(y-\alpha_i)^{m_i}=0 \}\subset \mathbb A^2_{y,u}$.
 where the $\alpha_i\in \mathbb C$ are distinct, $\sum m_i=6$, and moreover  $m_i$  is  odd if and only if $ i\le k$, namely we group  all the odd  $m_i$  at the beginning. 
 The singularities are the points $(\alpha_i,0)$ with $m_i>1$ and they are of type  $A_{n_i}$, $n_i:=m_i-1$. 
 The normalization of $U$ is $u^2-\Pi_1^k(y-\alpha_i)$, so it is reducible if and only if  all the $m_i$ are even.

 $(ii)$ We have already noticed in the proof of Proposition \ref{prop: types} that $C_1$ and $C_2$ have arithmetic genus 1 and meet in a point $p$ that is a node for $C$. The rest is an immediate consequence of Proposition \ref{prop: types} and Corollary  \ref{cor: bicanonical map}.
  \end{proof}

\begin{table}
 \caption{The possible curves of type A and B, with components marked with the geometric genus}
 \label{tab: list of curves}

\begin{tabular}[t]{ccc}
 \cmidrule[\heavyrulewidth](rl){1-2} \cmidrule[\heavyrulewidth](rl){3-3}
\multicolumn{2}{c}{ Type A} 
&
Type B  \\  \cmidrule(rl){1-2} \cmidrule(rl){3-3}
\begin{tikzpicture}
 \draw (-1,0) to[bend left] (1,0) node[above right] {2};
 \end{tikzpicture}
&
\begin{tikzpicture}
 \draw (-1.5,0) to[out = 30, in = 150] ++ (1,0) to[out = -30, in = 210, loop]  ++ (0,0) to[out = 30, in = 150] ++ (1,0) to[out = -30, in = 210, loop]  ++ (0,0) to[out = 30, in = 150] ++(1, 0)  node[above right] {0};
\end{tikzpicture}

&

\begin{tikzpicture}
 \draw (-.25,0) to[bend left] ++(2,0) node[above right] {1};
 \draw (-1.75,0)node[above left] {1} to[bend left] ++(2,0) ;
\end{tikzpicture}
\\

\begin{tikzpicture}
 \draw (-1,0) to[out = 30, in = 150] ++ (1,0) to[out = -30, in = 210, loop]  ++ (0,0) to[out = 30, in = 150] ++(1, 0)  node[above right] {1};
\end{tikzpicture}
&
\begin{tikzpicture}
 \draw (-1.5,0)  to[out = 30, in = 90] ++ (1,-.25) to[out = 90, in =150] ++ (1,.25) to[out = -30, in = 210, loop]  ++ (0,0) to[out = 30, in = 150] ++(1, 0)  node[above right] {0};
\end{tikzpicture}
&

\begin{tikzpicture}
 \draw (-.25,0) to[out = 30, in = 150] ++ (1,0) to[out = -30, in = 210, loop]  ++ (0,0) to[out = 30, in = 150] ++(1, 0)  node[above right] {0};
 \draw (-1.75,0)node[above left] {1} to[bend left] ++(2,0) ;
\end{tikzpicture}

\\
\begin{tikzpicture}
 \draw (-1,0) to[out = 0, in = -90] ++ (1,0.5)node[right] {\scriptsize $A_2$} to[out = -90, in = 180] ++(1, -.5)  node[above right] {1};
\end{tikzpicture}
&
\begin{tikzpicture}
 \draw (-1.5,0)  to[out = 30, in = 90] ++ (1,-.25) to[out = 90, in =90] ++ (1,0) to[out = 90, in =150]  ++(1, .25)  node[above right] {0};
\end{tikzpicture}
&
\begin{tikzpicture}
 \draw (-.25,0.1)  to[out = 0, in = -90] ++ (1,0.5) to[out = -90, in = 180] ++(1, -.5)  node[above right] {0};
 \draw (-1.75,0)node[above left] {1} to[bend left] ++(2,0) ;
\end{tikzpicture}
\\
\begin{tikzpicture}
\draw (1,.5) node[ right]{0} to[in = 0, out = 225] ++ ( -1,-.5) node[above]{\scriptsize $A_3$} to[in = 0, out = 180] ++ (-1,.5) to[out = 180, in= 180] ++ (0,-1) to[in = 180, out = 0] ++ (1,.5) to[out = 0, in  = 135] ++ (1, -.5);
 \end{tikzpicture}
& 

\begin{tikzpicture}
 \draw (-1.5,0)  to[out = 30, in = 210, looseness = 2] ++ (3,.5) node[ right] {0};
 \draw (-1.5,.5)  to[out = -30, in = -210, looseness = 2] ++ (3,-.5) node[ right] {0};
\end{tikzpicture}
&

\begin{tikzpicture}
 \draw (-.25,0) to[out = 30, in = 150] ++ (1,0) to[out = -30, in = 210, loop]  ++ (0,0) to[out = 30, in = 150] ++(1, 0)  node[above right] {0};
 \draw (-1.75,0)node[above left] {0} to[out = 30, in = 150] ++ (1,0) to[out = -30, in = 210, loop]  ++ (0,0) to[out = 30, in = 150] ++(1, 0);
\end{tikzpicture}

\\ 
\begin{tikzpicture}
 \draw (-1,0) to[out = 0, in = -90, looseness = 1.4] ++ (1,0.5) node[right]{\scriptsize $A_4$} to[out = -90, in = 180, looseness = 1.4] ++(1, -.5)  node[above right] {0};
\end{tikzpicture}

&
\begin{tikzpicture}
\draw (1,.5) node[ right]{0} to[in = 0, out = 225] ++ ( -1,-.5) node[above]{\scriptsize $A_3$} to[in = 0, out = 180] ++ (-1,.25) to[out = 180, in= 45] ++ (-.5,-.5);
\draw (1,-.5) node[ right]{0} to[in = 0, out = -225] ++ ( -1,.5) to[in = 0, out = 180] ++ (-1,-.25) to[out = 180, in= -45] ++ (-.5,.5);
\end{tikzpicture}
&

\begin{tikzpicture}
 \draw (-.25,0.1)  to[out = 0, in = -90] ++ (1,0.5) to[out = -90, in = 180] ++(1, -.5)  node[above right] {0};
 \draw (-1.75,0)node[above left] {0} to[out = 30, in = 150] ++ (1,0) to[out = -30, in = 210, loop]  ++ (0,0) to[out = 30, in = 150] ++(1, 0);
\end{tikzpicture}

\\

\multicolumn{2}{c}{
\begin{tikzpicture}
\draw (1,.5) node[ right]{0} to[in = 0, out = 225] ++ ( -1,-.5) node[above]{\scriptsize $A_5$} to[in = -45, out = 180] ++ (-1,.5);
\draw (1,-.5) node[ right]{0} to[in = 0, out = -225] ++ ( -1,.5) to[in = 45, out = 180] ++ (-1,-.5);
\end{tikzpicture}
}

& 
\begin{tikzpicture}
 \draw (-.25,0)  to[out = 0, in = -90] ++ (1,0.5) to[out = -90, in = 180] ++(1, -.5)  node[above right] {0};
 \draw (-1.75,0)node[above left] {0}  to[out = 0, in = -90] ++ (1,0.5) to[out = -90, in = 180] ++(1, -.5) ;
\end{tikzpicture}

\\
\cmidrule[\heavyrulewidth](rl){1-2} \cmidrule[\heavyrulewidth](rl){3-3}
\end{tabular}
\end{table}

\section{Generalised Gorenstein spin curves }
\label{sect: ggs}

A spin structure (or theta-characteristic)  on a smooth curve is a line bundle $\kl$ such that there is an isomorphism $\mu\colon\kl\otimes \kl \to \omega_C$. 
Here we propose a generalisation of this notion to Gorenstein curves. We remark that in the case of (semi-)stable curves there are several analogous notions in the literature (see e.g.,  \cite{Cornalba89,Jarvis98,Ludwig10}).

\begin{defin} 
Let $C$ be a projective  Gorenstein curve.
 A generalised Gorenstein spin (ggs) structure on $C$ is a pair $(\kl, \mu)$ where $\kl$ is a torsion-free sheaf of rank $1$ on $C$ with $\chi(\kl) = 0 $, and 
 \[ \mu \colon \kl\tensor_{\ko_C} \kl \to \kl\refl{2} \to \omega_C\]
 is a map which is an isomorphism in codimension zero, that is, at the generic point of each component of $C$. 
 We say the ggs structure is regular if $\kl$ is a line bundle.
 
 Isomorphisms of ggs structures are  defined in the obvious way, namely, ggs structures $(C_1, \kl_1,\mu_1)$  and $(C_1, \kl_1,\mu_1)$   
 are isomorphic if there are  an isomorphism $\phi\colon C_1\to C_2$ and   an isomorphism $\kl_1\to \phi^*\kl_2$ compatible with $\mu_1$ and $\phi^*\mu_2$. 
 In particular, when $\phi$ is the identity we may rescale $\mu$ by a non-zero scalar. 
  \end{defin}
 Note that the second condition implies that $\kl$ is generically locally free, since $\omega_C$ is locally free. This is  automatic on a reduced curve. 
 
Note also that for reasons of degree the multiplication map is an isomorphism if the ggs structure is regular. If $C$ is smooth,  all ggs structures are regular
because a  torsion-free sheaf is locally free at a smooth point of any curve.

In some applications ggs structures will appear naturally also on non-reduced curves but uniform results seem to be much harder to come by, so we do not address this case here.

The following is true if every partial normalisation of the Gorenstein curve $C$ is again Gorenstein, but we state it only for the case that we need.

\begin{prop}\label{prop: irregular ggs}
 Let $(\kl, \mu)$  be a ggs structure on a reduced Gorenstein curve with only singularities of type $A_m$. Then
 there exists a (uniquely determined)  partial normalisation $f\colon\tilde C \to C$ and a regular ggs structure  $(\tilde \kl, \tilde \mu)$  on $\tilde C$ such that 
 \[\kl = f_* \tilde \kl \text{ and }\kl\refl{2} = f_*\tilde\kl^{\tensor 2}.\]
\end{prop}

\begin{proof}
Since at a smooth point of a curve every torsion-free sheaf is free, the question is local at the singular points of $C$. 
Then the existence of the partial normalisation and the line bundle $\tilde \kl$ follows from  Lemma \ref{lem: pushforward}  its unicity from the description in Example \ref{example: A sing}. 

 For the last statement, fix a singular point $P\in C$ and denote by $Z$ its preimage in $\tilde C$. Note that it is possible to find an affine open set of $\tilde C$  containing  $Z$ on which $\tilde{\kl}$ is trivial. In fact,  twisting $\tilde{\kl}$ with a sufficiently ample effective divisor $H$ disjoint from $Z$  we may assume that $|\kl(H)|$ is base point free. So  there is a section $\sigma\in H^0(\tilde\kl(H))$ that does not vanish at any point of $Z$, and therefore   trivializes $\tilde \kl$ on an affine open set containing $Z$ and contained in the  complement of $H$.  Therefore  we may replace   $C$ by  affine open neighbourhood  $U=\spec A$ of $P$, $\tilde C$ by the preimage  $\tilde U=\spec\tilde A$  of $U$ in $\tilde C$ and $\tilde \kl$ by  the trivial bundle.  The map  $f_*\tilde \kl \otimes f_*\tilde \kl\to f_*\tilde\kl^{\otimes 2}$ corresponds to the map of $A$-modules $\tilde A\otimes_A \tilde A\to\tilde A$ induced by the multiplication of $\tilde{A}$, and thus it is surjective. 
 Since both sheaves have rank one and $f_*\tilde\kl$ is torsion-free, the induced map $\kl\refl{2}\to f_*\tilde\kl^{\tensor 2}$ is an isomorphism.

It remains to show that $(\tilde \kl, \tilde \mu)$    is a regular ggs structure on $\tilde C$. 
We have seen in Example \ref{example: A sing} that $\tilde C$ is again Gorenstein, so we can apply relative duality for a finite morphism of Gorenstein schemes, e.g., \cite[\href{https://stacks.math.columbia.edu/tag/0AU3}{Section 0AU3}]{stacks-project}, which gives
\[\Hom(\kl^{[2]}, \omega_C) = \Hom(f_*\tilde \kl^{\tensor 2}, \omega_C)\isom \Hom(\tilde \kl^{\tensor 2}, \omega_{\tilde C}),\]
and clearly $0 = \chi(\kl)  = \chi (\tilde \kl)$.
Restriction to the subset where $f$ is an isomorphism shows that $\phi \colon \kl\refl{2} \to \omega_C$ is generically an isomorphism if and only if the corresponding $\tilde \phi \colon \tilde \kl ^{\tensor 2} \to \omega_{\tilde C}$ is, completing the correspondence between ggs structures. 
\end{proof}
With Proposition \ref{prop: irregular ggs} in place, the enumeration of regular and irregular ggs structures on any given curve with type $A_n$  singularities is relatively simple, but the number of cases to consider grows quickly, if there are several possible partial normalisations to consider.

\begin{exam}
On a smooth rational curve, the only ggs structure  is $\kl\cong \ko_{\IP^1}(-1)$.
\end{exam}
\begin{prop}\label{prop: existence ggs}
  Let $C=\bigcup_i C_i$ be a reduced and connected Gorenstein curve, where $C_i$ are  the irreducible components.
  \begin{enumerate}
   \item The curve $C$ admits a regular ggs  structure $(\kl, \mu)$  if and only if  $\deg \omega_C|_{C_i}$ is even on every component $C_i$.
   \item If $C$ admits a regular ggs structure, then it admits exactly $2^{b_1}$ different regular ggs structures, where $b_1$ is the first Betti number of $C$.
  \end{enumerate}
In particular, every irreducible Gorenstein curve admits a regular ggs structure  while a curve with a separating node does not admit a regular ggs structure . 
\end{prop}
\begin{proof}
 Let $(\kl, \mu)$  be a regular ggs structure. Then for any $2$-torsion line bundle $\eta\in \Pic^0(C)$ the bundle $\kl \tensor \eta$ is a again a ggs structure and the set of regular ggs structures is a torsor under the group of $2$-torsion line bundles $\Pic^0(C)[2]$. By the exponential sequence, $\Pic^0(C)$ is a connected algebraic group that fits in an exact sequence 
  \begin{equation}\label{eq: exp seq}
 0 \to H^1(C, \IZ) \to H^1(C, \ko_C) \to \Pic^0(C) \to 0 
 \end{equation}
  and therefore  $2$-torsion elements in this group are given by the image of $\frac 12 H^1(C, \IZ)$ in $\Pic^0(C)$. This proves $(ii)$.
  
   If $\omega_C|_{C_i}$ has odd degree on any component, then there cannot exist a regular ggs structure $( \kl, \mu)$ on $C$, because $2\deg \kl|_{C_i} = \deg \omega_C|_{C_i}$. 
 
 On the other hand, assume that $\deg \omega_C|_{C_i} = 2d_i$ for all $i$ and pick a smooth point $p_i$ in $C_i$. Then $\eta :=\omega_C\left(-\sum_i 2d_i p_i\right)$ has degree zero on each component and thus $\eta\in \Pic^0(C)$. 
Since the group  $\Pic^0(C)$ is 2-divisible, for example by the description in \eqref{eq: exp seq}, we can pick an element $\eta'$ such that $\eta'^{\tensor 2} = \eta$. Then $\eta'\left( \sum_i d_i p_i \right)$ is a regular ggs structure on $C$.
\end{proof}

 We now collect the necessary information about ggs structures on the curves we are interested in. 
\begin{prop}\label{prop: ggs on genus 2}
 Let $C$ be a Gorenstein curve of arithmetic genus two with $\omega_C$ ample, bicanonical involution $\iota$ and ggs structure $(\kl, \mu)$. Then 
 \begin{enumerate}
  \item $\kl$ is invariant under the involution, that is, $\iota^*\kl \isom \kl$;
  \item $h^0(C, \kl)\leq 2$ and equality holds if and only if $C = C_1\cup C_2$ is of Type $B$ and $\kl = \ko_{C_1} \oplus \ko_{C_2}$ is the pushforward of the unique ggs  structure with a non-zero  section on each  component.
  \item $\mu\colon \kl\tensor\kl \to \omega_C$ is an isomorphism on the  locus  where $\kl$ is invertible. 
  \end{enumerate}
\end{prop}
\begin{proof}
 Let us write $\kl = f_* \tilde \kl$ for a regular ggs  structure $(\tilde \kl, \tilde \mu)$ on a partial normalisation $f\colon \tilde C \to C$. Recall that by Corollary  \ref{cor: description Types} the curve $\tilde C$ is also Gorenstein.
  The composition $f\circ \varphi: \tilde C \to Q$ is a degree 2 finite morphism, so it induces an involution 
 $\tilde \iota$ on $\tilde C$  that lifts $\iota$. 
 
 We first observe that $\iota$ acts trivially on $\Pic^0(\tilde C)[2]$. Indeed, considering the induced action on $H^1(\ko_{\tilde C}) \isom H^0(\tilde C, \omega_{\tilde C})^{\vee}$
   and taking into account that $\tilde C/\tilde \iota$ has arithmetic genus $0$ we see $\tilde \iota$ acts on $H^1(\tilde C, \ko_{\tilde C})$ as multiplication by $-1$.  Then the claim follows from the exponential sequence.
  It thus suffices to exhibit one regular  $\tilde \iota$-invariant ggs structure.

 By the criterion in Proposition \ref{prop: existence ggs} $(i)$ the curve $\tilde C$ admits a regular ggs structure if and only if it is irreducible or $\deg \omega_C|_{C_i}$ is even on both components.
 
  So if $C$ is of type B then $\tilde C$ is disconnected and  its connected components  are irreducible of  genus zero or one (cf.  Table \ref{tab: list of curves}), so we can take the trivial bundle on  components of genus one and the unique ggs structure on  components of genus zero and the resulting ggs structure is obviously invariant.

 Consider now $C$ of type A. By Corollary \ref{cor: description Types} if  $\tilde C$ is disconnected then it is the union of two smooth rational curves and we can take the unique ggs structure on each of them.
 If $\tilde C $ is connected but reducible, then it is   the union of  two smooth rational curves,  meeting either in two nodes or in an $A_3$ point, so it has genus 1 and we can take the trivial ggs structure
  $ \ko_{\tilde C}$   on $\tilde C$. 
 If $\tilde C\neq C$ is irreducible, then it has arithmetic genus at most one and admits either a trivial or a unique   regular ggs structure, as above.  
 
 So we are left with the case when $C=\tilde C$ is irreducible.
 If $\phi\colon C\to C/\iota = \IP^1$ admits a smooth branch point $p$, then the associated line bundle is a regular ggs structure, since $\omega_C = \phi^*\ko_{\IP^1}(1) = \ko_C(2p)$. The only remaining case is when $ \wt C$ has two points of type $A_2$.
 In this case, $b_1(C) = 0$ so by Proposition \ref{prop: existence ggs} there exists a unique ggs structure on $C$, which then has to be $\iota$-invariant.

 We move on to the second item.
 First observe that if  $C$ is a connected Gorenstein curve of genus $g\leq 2$, then for any regular  ggs structure $\kl$ on $C$ we have $h^0(C, \kl)\leq 1$, because $\deg \kl = g-1$.
 
  If $g = 2$ and $C$ is reducible, then by Proposition \ref{prop: existence ggs} and the geometry of the curves (Proposition \ref{prop: types}) any ggs structure is irregular, so we consider the corresponding regular ggs
  structure  $(\tilde \kl, \tilde \mu)$ on a partial normalisation $\tilde C$. If $\tilde C$ is connected, then  $h^0(\tilde \kl)\leq 1$ by the above. If $\tilde C$ is disconnected, then both components have genus at most one, so  $h^0(C, \kl)\leq 2$ with equality if and only if $C = C_1\cup C_2$ is of type $B$ and $\kl = \ko_{C_1} \oplus \ko_{C_2}$ is the pushforward of the unique ggs structure with a section on the components, because the unique regular ggs structure on $\IP^1$ does not have a section.
  
  $(iii)$ Follows from the fact that  the points where $\kl$ is invertible are those where $\tilde C\to C$ is an isomorphism  and from the fact that  $\mu$ is an isomorphism for regular ggs structures.
  \end{proof}

 We close this section with a technical result needed for the computations of Section~\ref{sec: half-can-rings}.
 \begin{lem}\label{lem: degL-B}
 Let $C$ be a reduced connected  Gorenstein curve of arithmetic genus two with $\omega_C$ ample and ggs structure $(\kl, \mu)$. Assume that $C$ is reducible and let $B\subset C$ be an irreducible component. Then $\deg(\kl\restr{B} ^{[1]})\ge -1$ and if equality holds then $p_a(B)=0$.
 \end{lem}
 \begin{proof} By Proposition \ref{prop: irregular ggs} we have  $\kl = f_* \tilde \kl$ for a regular ggs structure $\tilde \kl$ on a partial normalisation $f\colon \tilde C \to C$. Denote by $\tilde B$ the component of $\tilde C$ that maps to $B$ and denote by $\ki\subset \ko_{\tilde C} $ the ideal of $\tilde B$. Since $f$ is finite, we have an exact sequence 
 \begin{equation}
 0\to f_* 	(\ki \tilde\kl) \to f_*\tilde \kl=\kl \to f_*(\tilde \kl \restr{\tilde B})\to 0
 \end{equation}
 The sheaf $f_*(\tilde \kl \restr{\tilde B})$ is torsion-free and supported on $B$,  so the last map in the above equation gives a surjective map $ \kl\restr{B}^{[1]}\to f_*(\tilde \kl \restr{\tilde B})$. Since both sheaves are torsion-free of rank one, the map is an isomorphism. 
 
 So we have $\chi( \kl\restr{B}^{[1]})=\chi({\tilde\kl\restr{\tilde B})=\deg\tilde \kl\restr{\tilde B}+1-p_a(\tilde B})$ and $2\deg (\tilde\kl\restr{\tilde B})=\deg K_{\tilde C}\restr{\tilde B}$. Since there is an inclusion $\omega_{\tilde B}\into \omega_{\tilde C}$ (cf. Lemma \ref{lem:  sequence-for-omega}),  we get $\deg\tilde \kl \restr{\tilde B}\ge p_a(\tilde B)-1 \ge -1$, and if $\deg\tilde \kl \restr{\tilde B}=-1$ holds  then  $\tilde B$ is smooth rational. 
   
 Assume $\deg \tilde \kl\restr{\tilde B}=-1$, hence $0=\chi( \tilde \kl\restr{\tilde B}) =\chi( \kl\restr{B}^{[1]})$: if $B$ is not smooth rational then $p_a(B)=1$ and $\deg  \kl\restr{B}^{[1]}=0$.  \end{proof}

 \begin{cor}\label{cor: degL+K_C}
 Let $C$ be a reduced connected Gorenstein curve of arithmetic genus two with $\omega_C$ ample and ggs structure $(\kl, \mu)$.  Then  we have
 $H^1(C, \kl \otimes \omega_C^{\otimes m})=0$ for every $m\geq 1$. 
 \end{cor}
\begin{proof}
We have  $\deg (\kl \otimes \omega_C^{\otimes m}) =2m+1\ge 3=2p_a(C)-1$, since $\omega_C$ is a line bundle of degree two. Therefore if $C$ is irreducible we can conclude by Lemma \ref{vanishingH^1}.   

 Assume now  that $C$ is reducible and let $B\subset C$ be an irreducible component, so that $\omega_C\restr B$ is a line bundle of degree one.
 Then by the above Lemma \ref{lem: degL-B}  we have $\deg((\kl \otimes \omega_C^{\otimes m})\restr{B}^{[1]})\ge m-1$  if $p_a(B)=0$ and $\geq m$ if $p_a(B)=1$.  Summing up,  we have $\deg (\kl \otimes \omega_C^{\otimes m})\restr{ B}^{[1]} \geq 2p_a(B) -1$ for every $B\subseteq C$, 
 hence   Lemma \ref{vanishingH^1} gives  $H^1( \kl \otimes \omega_C)=0$.
\end{proof}

\section{Half-canonical rings}\label{sec: half-can-rings}

\subsection{Definition and classification of half-canonical rings}

\begin{defin} Let $C$ be a projective and reduced Gorenstein curve and let 
 $( \kl, \mu)$ be a ggs structure on $C$.  Its  half-canonical ring  is 
 \[ R(C, \{\kl, \omega_C\}) = R(C, \kl) = \bigoplus_n R_n, \]
 where $R_{2n} = H^0(C, nK_C) $ and $R_{2n+1} = H^0(C, \kl(nK_C))$ and  the multiplication of two elements of odd degree is defined via $\mu$. 
\end{defin}
Now let $( \kl, \mu)$ be a ggs  structure on a  curve  $C$ with ample canonical bundle and arithmetic genus two. 
We classify the possible half-canonical rings $R(C,\{\kl,\omega_C\})$ according to the type of the curve defined in  Proposition \ref{prop: types}.

\begin{thm}\label{thm: all half canonical rings}
  Let $( \kl, \mu)$  be a ggs structure on a curve $C$ of genus two with ample canonical bundle. 
Then $h^0(\kl) = 0, 1, 2$ and  the half-canonical ring $R(C, \{\kl, \omega_C\})$ admits generators and relations as given in Table \ref{tab: half-canonical rings}. 

 Conversely, under the assumptions  on $f_6$, $g_6$, $g_8$, $f_{12}$ shown in  Table \ref{tab: half-canonical rings},    the  ideals of Table \ref{tab: half-canonical rings} define  reduced 
 Gorenstein curves  $C$ of genus two with $\omega_C = \ko_C(2)$ ample
 and $\kl = \ko_C(1)\refl{1}$  a ggs structure on $C$. 
\end{thm}

\begin{table}
\caption{Half-canonical rings in terms of generators and relations}
\label{tab: half-canonical rings}
 \centering
\begin{tabular}{cccc}
 \toprule
 $h^0(\kl)$ &Type & \begin{minipage}[b]{2cm} variables \\ \& degrees\end{minipage} &  relations\\
 \midrule
 0 & $A(0)$  & $\mat{y_1 & y_2 & z_1 & z_2\\ 2 & 2& 3& 3}$ & $z_1^2 - f_{6}(y_1, y_2),\,z_2^2-g_6(y_1, y_2)$\\
 & & & where $f_6 \neq 0$, $g_6 \neq 0 $
 \\
 \cmidrule{2-4}
 &$B(0)$ & $\mat{y_1 & y_2 & z_1 & z_2 & v & t_1 & t_2 & u\\ 2 & 2& 3& 3 & 4 &  5& 5 & 6}$ 
 &
 \begin{minipage}{.4\textwidth}
 \[\rk
  \begin{pmatrix}
y_1 & 0 & z_1 & t_1 \\
0 & y_2 & z_2 & t_2 \\
z_1 & z_2 & v & u \\
t_1 & t_2 & u & g_8
\end{pmatrix}
\leq 1
\]
where $g_8=g_8(y_1,y_2,v)$ and $v^2$ appears in $g_8$ with non-zero coefficient. 
 \end{minipage}
\\
\midrule
1 &$A(1)$
& 
$\mat{x & y & z\\ 1 & 2 & 5}$ 
& 
$z^2 - f_{10}(x,y)$
\\ 
& & & where $f_{10}\neq 0$\\ 
\cmidrule{2-4}
& $B(1)$
&
$\mat{x& y& w& v& z& u\\ 1& 2& 3& 4& 5& 6}$ 
& 
\begin{minipage}{.4\textwidth}
\[\rk \begin{pmatrix}0&y&w&z\\x&w&v&u\end{pmatrix}\leq 1,\] 
\[
\begin{array}{rcl}
z^2 & = & yg_8(y,v) \\
zu & = & wg_8(y,v) \\
u^2 & = & vg_8(y,v) + x^4h_8(x,v)
\end{array}
\]

where $v^2$ appears in $g_8$ with non-zero coefficient. 
\end{minipage}

\\
\midrule
2&$B(2)$
&
$\mat{ x_1& x_2 & v & u \\ 1&1& 4& 6}$
&
\begin{minipage}{.4\textwidth}
\[x_1x_2,\ \  u^2 - f_{12}(x_1, x_2, v)\]
where $v^3$ appears in $f_{12}$ with non-zero coefficient.
\end{minipage}
\\
\bottomrule
 \end{tabular}

\end{table}

\begin{rem}
We have the following geometric interpretations of the type A rings:
\begin{enumerate}
\item If $h^0(\kl)=1$ then $C$ is a double cover of $\IP^1$ with coordinates $x^2, y$ with ramification divisor $R$ consisting of the five points determined by $f_{10}$ together with the point $p$ defined by $x^2 = 0$.
 If $C$ is smooth at $p$, then $\kl=\ko_C(p)$. 
\item If $h^0(\kl)=0$ then $C$ is a double cover of $\IP^1$ with coordinates $y_1,y_2$ with ramification  divisor $R$ 
 consisting of the three points $p_1,p_2,p_3$ determined by $f_6$ and the three points $q_1,q_2,q_3$  determined by $g_6$. If $C$ is smooth, then $\kl=\ko(p_1+p_2-p_3)=\ko(q_1+q_2-q_3)$. 
\end{enumerate}
In both cases, singularities of $C$ correspond to multiple points of the branch locus.
\end{rem}
\begin{rem}The type B curves $C=C_1\cup C_2$ are double covers of the reducible conic $Q_1\cup Q_2$ branched on the point $Q_1\cap Q_2$   and on three more  points  on each component, cf.~Corollary~\ref{cor: bicanonical map}. The geometry of the half-canonical rings is as follows:
\begin{enumerate}
\item If $h^0(\kl)=0$ then one component is a double cover of $\IP^1$ with coordinates $y_1^2,v$ branched over the points $v=0$ and  $y_1=0$  and the two points determined by $g_8|_{y_2=0}$. 
The other component is a double cover of $\IP^1$ with coordinates $y_2^2,v$ branched on  the points $v=0$ and $y_2=0$and the two points determined by $g_8|_{y_1=0}$.
\item If $h^0(\kl)=1$ then one component is a double cover of $\IP^1$ with coordinates $x^4,v$ branched in  $[0,1]$  and  in the three points determined by $v^3+x^4h_8$. On the other component, the double cover of $\IP^1$ can be viewed in coordinates $y^2,v$, branched in the point $v=0$ and the two points determined by $g_8$.
\item If $h^0(\kl)=2$ then one component is a double cover of $\IP^1$ with coordinates $x_1^4,v$ branched   in  $[0,1]$   and   in the three points determined by $h_{12}|_{x_2=0}$. 
The other component is a double cover of $\IP^1$ with coordinates $x_2^4,v$ branched in in  $[0,1]$   and  the three points determined by $h_{12}|_{x_1=0}$.
\end{enumerate}
\end{rem}

We prove Theorem \ref{thm: all half canonical rings} case-by-case, starting from the description  of the canonical ring 
\[ R(C, \omega_C)  \subset R =  R(C, \{\kl, \omega_C\})\]
obtained in Proposition \ref{prop: types}
and, for type B,  exploiting the involution $\iota$ on $C$ induced by the bicanonical map.

\subsection{First observations}

We begin by showing that the Hilbert series of $R$ is determined by the value  of $h^0(\kl)$. 
 \begin{prop}\label{prop: extended  involution} 
 Let $( \kl, \mu)$  be a ggs  structure on a curve of genus two with ample canonical bundle. 
Then the Hilbert series of the half-canonical ring $R$ is 
 \begin{center}
  \begin{tabular}{cc}
    $h^0(C, \kl)$& Hilbert series of $R$\\
   \midrule 
   0 &  $\frac{(1-t^6)(1-t^6)}{(1-t^2)(1-t^2)(1-t^3)(1-t^3)}$\\ 
    1&  $\frac{1-t^{10}}{(1-t)(1-t^2)(1-t^5)}$\\
   2 &  $\frac{(1-t^2)(1-t^{12})}{(1-t)(1-t)(1-t^4)(1-t^6)}$ \\
\bottomrule
\end{tabular}
 \end{center}
 \end{prop}
\begin{proof} 
 The dimension of $R_2=H^0(K_C)$ is two  and for $m>1$ we know  $\dim R_{2m}=h^0(C,mK_C) = 2m-1$  from Proposition \ref{prop: types}, so we only need to determine the dimension of $R_{2m+1}=H^0(\kl(mK_C))$.

 The dimension of $R_1=H^0(\kl)$ is fixed by our assumption. 
 Concerning the other odd degrees, by Corollary \ref{cor: degL+K_C} we have   $H^1(C, \kl(mK_C))=0$, whence by  Riemann--Roch we obtain $h^0(C, \kl(mK_C))= 2m $ for every $m\geq 1$.

    Putting this all together as in the proof of Proposition \ref{prop: types} gives the stated Hilbert series for the ring $R$. \end{proof}

The next result is used to bound our computations of half-canonical rings.
\begin{lem}\label{lem: generated in low degree}
\begin{enumerate}
\item If $( \kl, \mu)$  is a ggs structure on a curve of Type A then $R(C,\{\kl,\omega_C\})$ is generated in degrees $\le 6$.
\item If $( \kl, \mu)$  is a ggs structure on a curve of Type B then $R(C,\{\kl,\omega_C\})$ is generated in degrees $\le 9$.
\end{enumerate}
\end{lem}
\begin{proof}
$(i)$ By Proposition \ref{prop: types},  if $C$ is of Type A, then $|K_C|$ is base point free.  Hence   by Proposition \ref{prop: surjectivity} and 
 Corollary \ref{cor: degL+K_C} the following multiplication maps 
 $$H^0(K_C)\tensor H^0(mK_C) \to H^0((m+1)K_C) = R_{2m+2}, $$ $$ \ H^0(K_C)\tensor H^0\left(\kl ((m-1)K_C) \right) \to H^0(\kl(mK_C)) = R_{2m+1}$$   
 are surjective for  $m\ge 3$. Hence $R$ is generated in degrees $\le 6$. 

$(ii)$ Since the linear system
  $|2K_C|$ is base point free and since for $m\geq 3$  Corollary~\ref{cor: degL+K_C} gives  $H^1(C,\kl((m-2)K_C))=0$,
  Proposition \ref{prop: surjectivity} applies to  $\kl(mK_C)$  and $2K_C$ and so the multiplication map 
 $$H^0(2K_C)\tensor H^0(\kl(mK_C)) \to H^0(\kl((m+2)K_C)) = R_{2m+5}$$
 is surjective for all $m\ge3$. It follows that $R$ is generated in degrees $\leq 10$. Since the canonical ring is generated in degrees $\le 6$ (Proposition \ref{prop: types}), we also know that there are no new generators in degree $10$.
\end{proof}

\subsection{Type A}

Let $C$ be a curve of type A, i.e.,  by Corollary  \ref{cor: description Types},
$C$ is irreducible or $C = C_1\cup C_2$ with $p_a(C_i)=0$.
We say that a  ggs structure $( \kl, \mu)$  is of type $A(i)$ if  $C $ is of type $A$ and $h^0(C,\kl)=i$. 
 By Proposition \ref{prop: ggs on genus 2} we have $i =0$ or $i = 1$. 

We are now going to prove cases $A(1)$ and $A(0)$ of Theorem  \ref{thm: all half canonical rings}.

\subsubsection{Type $A(1)$}
Here we prove that if $C$ is of type $A(1)$ then $R(C,\{\kl,\omega_C\})$ is isomorphic to $\IC[x,y,z]/(z^2 - f_{10})$, with $\deg(x,y,z)=(1,2,5)$.
\begin{proof}
Choose a non-zero section $x\in R_1 = H^0(\kl)$. Since $C$ is reduced we have $x^2 \neq 0$ and thus we can choose the coordinates in the canonical ring such that $x^2 = y_1$ and let $y : = y_2$. 
We start by noting the following facts:
\begin{itemize}
\item[(a)] Multiplication by $x$ gives an injective map $R(C,\{\kl,\omega_C\})\to R(C,\{\kl,\omega_C\})$. Indeed, we have seen in the proof of 
Proposition \ref{prop: types} that a curve of type A  either is irreducible, or it is the union of two smooth rational curves $C_1$ and $C_2$ and the restriction map $H^0(\omega_C)\to H^0(\omega_C\restr{C_i})=H^0(\OO_{C_i}(1))$ is an isomorphism for $i=1,2$. So $x^2$, and a fortiori $x$, does not vanish at the generic point of any component of $C$. 

\item[(b)] The map $j\colon \IC[x, y]\to R(C,\{\kl,\omega_C\})$ is injective. Let $v$ be a nonzero  element of $\ker j$: since $j$ is a map of graded rings, we may assume that $v$ is homogeneous. Then $\deg v$ is odd, since the map  $\IC[x^2,y]\to R(C,\omega_C)$  is an inclusion by Proposition \ref{prop: types} and we can write $v=xv_1$ with $v_1$  homogeneous of even degree. By (a) $v_1$ is in the kernel of  $\IC[x^2,y]\to R(C,\omega_C)$, a contradiction.
\end{itemize}
In view of  fact (b) above, comparing dimensions we see that $\IC[x,y]_{\leq 4} = R_{\leq 4}$ and that
\[ R_5 = \langle x^5, x^3y, xy^2, z\rangle\]
for some section $z$. 
The graded piece $R_6$ contains $xz$ and the four monomials of degree 6 in $x,y$. Assume by contradiction that there is a  relation among these elements, so that $xf_5(x,y,z)=\lambda y^3$ for some $\lambda\in \IC$. If $\lambda=0$, then $f_5=0$ by fact (a) and the linear relation is trivial. If $\lambda\ne 0$, then squaring the relation we obtain that $x^2$ divides $y^6$, but this is impossible since $y$ and $x^2$ have no common zero on $C$.
So we see that we have no new generators in degree $d=6$,  hence by Lemma \ref{lem: generated in low degree}, we have found all generators of $R$.

 Let $u\in R_6$ be as in  Proposition \ref{prop: types}.  We can write $u=xw+\lambda y^3$ with $w\in R_5$ and $\lambda\in \IC$. Since $u \notin \IC[x,y]$ as in the proof of Proposition \ref{prop: types} we have  $w= \alpha z+ p(x, y)$, with $\alpha\in \IC^*$, so up to changing $z$ we may assume $u=xz+\lambda  y^3$.

Now we study the relations. By Proposition \ref{prop: types},  
we may write $z^2 = f_{10}(x^2,y) + u f_4(x^2,y)\in R_{10}$. Combining this with the relation $u^2-f_{12}(x^2,y)=0$ in the canonical ring we have
\[ x^2 f_{10}(x^2, y) + ux^2 f_4(x^2, y) = x^ 2 z^2 = \lambda^2 y^6- 2\lambda  y^3u + u^2 = \lambda^2 y^6- 2\lambda u y^3+  f_{12}(x^2, y).\]
A  basis of   $R_{12}$ is $x^{12}, x^{10}y, \dots, y^6, x^6u, x^4yu, x^2y^2u, y^3 u$, so 
comparing  the coefficients of the monomials involving $u$  we see that $\lambda = 0 $, $f_4=0$,  $xz = u$ and    
we have the claimed relation $z^2 - f_{10} =0$.
Comparing the Hilbert series (cf. Proposition \ref{prop: extended involution}), we see that we have  found the full ring.

\end{proof}

\subsubsection{ Type $A(0)$}
Now we prove that if $C$ is of type $A(0)$ then \[R\isom \IC[y_1,y_2,z_1,z_2]/(z_1^2 - f_{6},z_2^2-g_6),\] with $\deg(y_1,y_2,z_1,z_2)=(2,2,3,3)$.

\begin{proof}
Since $R_1=H^0(\kl)=0$, we start with the basis $y_1,y_2$ of $R_2=H^0(\omega_C)$ coming from the canonical ring of $C$. 
We read off $\dim R_3=h^0(\kl(K_C))=2$ from the Hilbert series of $R$ (cf. Proposition \ref{prop: extended involution}) and choose a basis $z_1,z_2$. Since $\IC[y_1,y_2]$ is a subring of $R$, we have $\IC[y_1,y_2,z_1,z_2]_{\leq 4}=R_{\leq 4}$. Since $h^0(\kl)=h^1(\kl)=0$ and $|K_C|$ is a free pencil, the multiplication  $R_2\otimes R_3\to R_5$ is a bijection by Proposition \ref{prop: surjectivity},  so 
\[R_5=\langle y_1z_1,y_2z_1,y_1z_2, y_2z_2\rangle.\]
From the canonical ring, $R_6$ is generated by cubic polynomials  in $y_1,y_2$ and by  $u$, where $u$ is as in Proposition \ref{prop: types}. Hence by Lemma \ref{lem: generated in low degree} we have found all generators.

 Thus the elements of $\Sym^2(z_1,z_2)$ can be expressed in terms of $y_1,y_2,u$, namely there are three relations of  the  form
\begin{equation*}
z_1^2 =\alpha u +a_6(y_1,y_2),\quad
z_1z_2= \beta u +b_6(y_1,y_2), \quad
z_2^2 = \gamma u +c_6(y_1,y_2)
\end{equation*}
Comparing the Hilbert series (cf. Proposition \ref{prop: extended involution}) we see that there are no additional relations
 and we can embedd $C$ into $\IP(2,2,3,3,6)$.  Assume that $\alpha=\beta=\gamma=0$: then $C$ contains the point $q=[0,0,0,0,1]$ and for a general choice
  of $a_6,b_6,c_6$ the tangent cone to $C$ at this point has six irreducible components
 since  the 3 equations give $6$ points in $\pp(2,2,3,3)$ and the curve is the cone over these 6 points with vertex $q$, 
  so $C$ does not have an $A_n$ singularity at $q$, contradicting Corollary \ref{cor: description Types}. Therefore at least one of $\alpha, \beta, \gamma$ is nonzero and we can use one of the relations above to eliminate $u$. Up to a linear change of the coordinates $z_1,z_2$  the remaining relations can be written either as:
\begin{equation}\label{eq: f6g6}
z_1^2 =f_6(y_1,y_2),\quad
z_2^2 = g_6(y_1,y_2)
\end{equation}
or
\begin{equation*}
z_1^2 =f_6(y_1,y_2),\quad
z_1z_2= h_6(y_1,y_2). 
\end{equation*}
In the latter case, $C$ contains the point $[0,0,0,1]$ and again  one can show that the corresponding singularity is not of type $A_n$. So the relations are as in \eqref{eq: f6g6}. 
\end{proof}

\begin{rem}\label{rem: linearisation A}
If $( \kl, \mu)$  is a ggs structure on a curve $C$ of type $A$ then it is possible to extend the action  of the bicanonical involution $\iota$ on the canonical ring (see Corollary \ref{cor: bicanonical map}), to the half-canonical ring  $R(C,\{\kl,\omega_C\})$. In case $A(1)$ we take 
$$(x,y,z)\mapsto (x,y,-z)$$ and in case $A(0)$ we take
\[(y_1,y_2,z_1,z_2)\mapsto (y_1,y_2,z_1,-z_2).\]
This amounts to choosing a linearisation of $\kl$ compatible with the one chosen for $\omega_C$ (cf. Remark \ref{rem: linearisation}). 
 Note that these are not the only possible choices: in the former case, we can also consider $(x,y,z)\mapsto (-x,y,-z)$ and in the latter case $(y_1,y_2,z_1,z_2)\mapsto (y_1,y_2,-z_1,z_2)$.\end{rem}

\subsection{Type  B: talking 'bout an involution}
Now let $C$ be a canonical curve of type B, i.e., 
$C = C_1\cup C_2$ is a reducible and reduced  Gorenstein curve of arithmetic genus two such that the two components have arithmetic genus one and intersect in a node $p$ (see Corollary~ \ref{cor: description Types}). By Corollary~\ref{cor: bicanonical map}, the bicanonical map $\varphi\colon C\to Q=Q_1 \cup Q_2$ has degree $2$ with $Q_i\isom \IP^1$, and the bicanonical involution $\iota$ is regular. 
 We say  a  ggs structure $( \kl, \mu)$  is of type $B(i)$ if  $C$ is of type $B$ and $h^0(C,\kl)=i$. 
 By Proposition \ref{prop: ggs on genus 2} we have $i =0, 1, 2$. 
 
 Recall that we have chosen a linearisation on $\omega_C$ with respect to the bicanonical  involution $\iota$ (Remark \ref{rem: linearisation}). In case $A$ we have determined all the possible half-canonical rings and a posteriori we have checked that the action of $\iota$ on the canonical ring extends to the half canonical rings. In other words, we have shown  that every ggs structure $( \kl, \mu)$  on a curve of type $A$  is linearisable and that the map $\mu\colon \kl\otimes \kl \to\omega_C$ is equivariant with respect to the fixed linearization of $\omega_C$. In case $B$ we proceed in a different order: we first show the existence 
 of a linearisation such that the map $\mu\colon \kl\otimes \kl \to\omega_C$ is equivariant and then we use the induced action of $\iota$ on $R(C,\{\kl,\omega_C\})$ to analyse the ring structure.
 
 \begin{lem}\label{lem: linearisation B}  Let $C$ be a Gorenstein curve of genus two with $\omega_C$ ample of type $B$ and let $( \kl, \mu)$ be a ggs structure on $C$.
 Then: 
 \begin{enumerate} 
 \item $\kl$ can be linearised with respect to the bicanonical involution $\iota$ of $C$
 \item  for any choice of linearisation of $\kl$   the map $\mu\colon \kl \otimes \kl \to \omega_C$ is equivariant with respect to the linearisation  of $\omega_C$ fixed in Remark \ref{rem: linearisation}.
 \end{enumerate}
 \end{lem}
\begin{proof}
 By Proposition \ref{prop: irregular ggs} there are a partial normalisation $f\colon \tilde C\to C$ and a  regular ggs  structure  $(\tilde \kl, \tilde \mu)$  on $\tilde C$ such that $\kl=f_*\tilde \kl $. 
 
$(i)$ In the proof of Proposition \ref{prop: ggs on genus 2}, $(i)$ we have seen that $\iota$ induces an involution $\tilde \iota$ of $\tilde C$ such that  $\tilde\iota^*\tilde \kl\cong \tilde\kl$. Since $\tilde \kl$ is a line bundle, this implies that $\tilde\kl$ admits a linearisation. Pushing forward to $C$ we obtain a linearisation of $\kl$.  

$(ii)$  
Linearisations  of $\kl$, $\tilde \kl$ with respect to $\iota$, $\tilde \iota$ can be seen to be  in bijection arguing as in the proof of Lemma \ref{lem: pushforward}, $(ii)$.
By relative duality there is a  $\ko_C$-module  homomorphism $f_*\omega_{\tilde C}\to \omega_C$ which is invariant with respect to the natural linearisations of these sheaves  with respect to   $\tilde{\iota}$ and $\iota$ (cf. Remark \ref{rem: linearisation}). 				 
 So it is enough to show that for any given linearisation of $\tilde\kl$,  the isomorphism $\tilde\kl \otimes \tilde\kl\to \omega_{\tilde C}$ induces on $\omega_{\tilde C}$ the opposite of the natural linearisation. 
  
  By \cite[Lemma 4.4]{FFP22} there is a trivializing open cover  $\{ U_j\}$  for $\tilde \kl$ on   $\tilde C$ such that the  $U_j$ are $\tilde \iota$-invariant. Let $\sigma_j$ be a local generator of $\tilde\kl \restr{U_j}$ and let $q\in U_j$ 
   be a point: then at least one between $\sigma_j+{\tilde\iota}^*\sigma_j$ and $\sigma_j-{\tilde\iota}^*\sigma_j$ does not vanish at $q$. 
  So, up to refining the cover, we may assume in addition that for every $j$  the sheaf $\tilde\kl \restr{U_j}$ is  generated by an element $\tau_j$ such that ${\tilde\iota}^*\tau_j=\pm \tau_j$.
  Via the isomorphism $\tilde\kl\tensor \tilde \kl\to \omega_{\tilde C}$ we obtain a local trivialization of $\omega_{\tilde C}$ such that  for every $j$ there is a local generator $\alpha_j$ on $U_j$  that  verifies ${\tilde \iota}^*\alpha_j=\alpha_j$. In other words, $\omega_{\tilde C}$ is the pull back of a line bundle $\tilde \km$ on  $\tilde C/\tilde \iota$ and the linearisation induced by the ggs structure coincides with the linearisation as the  pullback of $\tilde \km$.
  By Proposition \ref{prop: existence ggs} $\kl$ is not regular at the intersection point of $C_1$ and $C_2$, hence $\tilde C=\tilde C_1\sqcup \tilde C_2$ is a disjoint union of curves of arithmetic genus 0 or 1.  
   The bicanonical involution $\iota$ induces an  involution $\tilde\iota$ of $\tilde C$ that preserves the two components and  such that $Q_i:=\tilde C_i/\tilde \iota$ is smooth rational for $i=1,2$. 
   So the quotient map $\phi_i\colon \tilde C_i\to Q_i$ is flat and ${\phi_i}_*\ko_{\tilde C_i}=\ko_{Q_i}\oplus \ko_{Q_i}(-a)$, where $a=p_a(C_i)+1$. 
   Relative duality gives ${\phi_i}_*\omega_{\tilde C_i}=\omega_{Q_i}\oplus\omega_{Q_i}(a)$. Since  $\deg \omega_{\tilde C_i}=2\deg \km\restr{Q_i}$, we conclude that $\km\restr{Q_i}=\omega_{Q_i}(a)$,
    namely the linearisation induced on $\omega_{C_i}$ by the isomorphism $\tilde\kl\tensor \tilde \kl\to \omega_{\tilde C}$ is  the opposite of the natural one. 
     \end{proof}

\begin{lem}\label{lem: B decomposition}
Let $( \kl, \mu)$  be a ggs structure  of   type $B(i)$ with
$i=h^0(\kl)=a_1+a_2$ where $(a_1,a_2)=(0,0)$, $(0,1)$ or $(1,1)$.  
Then, if we use the linearisation of $\omega_C$ chosen in Remark \ref{rem: linearisation}:
\begin{enumerate}

\item there is a unique linearisation of $\kl$ such that  both $R_3=H^0(\kl(\omega_C))$ and $R_1=H^0(\kl)$ are $\iota$-invariant; 
\item for the above choice of linearisation we have the following splittings  under the action of $\iota$:

\end{enumerate}
  
\begin{description}[font=\normalfont]
\item[If $n>0$ is odd]
\begin{align*}R_{2n+1}^+ &\isom H^0\left(\ko_{Q_1}{\left( \tfrac {n-1}2 \right)}\right) \oplus H^0\left(\ko_{Q_2}{\left( \tfrac {n-1}2 \right)}\right) \\
 R_{2n+1}^-&\isom H^0\left(\ko_{Q_1}{\left( \tfrac {n-1}2 -1\right)}\right) \oplus H^0\left(\ko_{Q_2}{\left( \tfrac {n-1}2 -1\right)}\right) 
\end{align*}

\item[If $n\ge 0$ is even]

\begin{align*} R_{2n+1}^+&\isom H^0\left(\ko_{Q_1}{\left( \tfrac n2 -1+a_1\right)}\right) \oplus H^0\left(\ko_{Q_2}{\left( \tfrac {n}2-1+a_2 \right)}\right)\\
 R_{2n+1}^-&\isom H^0\left(\ko_{Q_1}{\left( \tfrac n2 -1-a_1\right)}\right) \oplus H^0\left(\ko_{Q_2}{\left( \tfrac n2 -1-a_2\right)}\right) 
 \end{align*}
 \end{description}

\end{lem}

\begin{proof}

Since $\kl$ is not locally free
  at the intersection point $p$ of $C_1$ and $C_2$ (cf.~Proposition~\ref{prop: existence ggs}), by Proposition \ref{prop: irregular ggs} there is a partial normalisation $f\colon \tilde C_1\sqcup \tilde C_2\to C$ such that 
$\kl=f_*\tilde \kl$, where $\tilde \kl =\tilde \kl_1 \oplus \tilde \kl_2$ and $\tilde \kl_i$ is a line bundle supported on $\tilde C_i$. Recall that $\tilde C_i$ is Gorenstein, $p_a(\tilde C_i)=0$ or 1 and $\tilde \kl_i^{\otimes 2}=\omega_{\tilde C_i}$. We set $\kl_i:=f_*\tilde\kl_i$ and $a_i:=h^0(\kl_i)=h^0(\tilde \kl_i)$. We have $a_1+a_2=h^0(\kl)\le 2$ by  Proposition \ref{prop: ggs on genus 2}, so if we assume $a_1\le a_2$ the possibilities for the pair $(a_1,a_2)$ are $(0,0)$, $(0,1)$ or $(1,1)$.

As we have seen  in the proof of Proposition~\ref{prop: ggs on genus 2}, $\iota$ induces an involution $\tilde \iota$ on $\tilde C_1\sqcup \tilde C_2$ preserving the two components and with $\tilde C_i/\tilde \iota= Q_i \isom \IP^1$. 
If $\varphi_i \colon \tilde C_i \to Q_i$ is the quotient map, i.e., $\varphi_i$ is the restriction to $\tilde C_i$ of the composition $\varphi\circ f\colon \tilde C_1\sqcup \tilde C_2\to Q_1\cup Q_2$, 
then we have
\begin{align*}
\omega_C^{\otimes n} \otimes \kl =\omega_C^{\otimes n} \otimes (\kl_1 \oplus \kl_2)=(\omega_C^{\otimes n} \otimes \kl_1) \oplus (\omega_C^{\otimes n}\oplus \kl_2)=\\
(\omega_C^{\otimes n}\restr{C_1} \otimes \kl_1)  \oplus 
(\omega_C^{\otimes n}\restr{C_2} \otimes \kl_2) 
=\kl_1\otimes \ko_{C_1}(np) 
\oplus
\kl_2\otimes \ko_{C_2}(np).
\end{align*}
Therefore
\[
R_{2n+1}=H^0(\omega_C^{\otimes n} \otimes \kl)\isom H^0 \left( {\varphi_1}_*{\tilde \kl_1(np_1)} \right)
\oplus
H^0 \left( {\varphi_2}_*{\tilde\kl_2(np_2)} \right).
\]
where $p_1,p_2$ are the preimages of the node $p=C_1 \cap C_2$,
In turn, each summand $H^0 \left( {\varphi_i}_*{\tilde\kl_i(np_i)} \right)$ decomposes into eigenspaces under the action of $\iota$.
Since $\ko_{C_i}(2p_i)=\phi_i^*\ko_{Q_i}(1)$, the projection formula gives 
\[ {\varphi_i}_*{\tilde\kl_i(np_i)}=\begin{cases}
                               ({\phi_i}_*\tilde\kl_i)\otimes \ko_{Q_i}(\tfrac n 2) & \text{ $n$ even,}\\
                                ({\phi_i}_*\tilde\kl_i(p_i))\otimes \ko_{Q_i}(\tfrac {n-1}2) & \text{ $n$  odd.}
                              \end{cases} 
\] In addition, $\ko_{C_i}(2p_i)$ has the pull-back linearisation, so it is enough to determine  the decomposition into eigensheaves of ${\phi_i}_*\tilde\kl_i$ and  ${\phi_i}_*\tilde\kl_i(p_i)$.

The line bundle $\tilde\kl_i(p_i)$ has degree $p_a(\tilde C_i)$. If $p_a(\tilde C_i)=0$, then  $\tilde\kl_i(p_i)=\ko_{\tilde C_i}$ and ${\phi_i}_*\tilde\kl_i(p_i)=\ko_{Q_i}\oplus \ko_{Q_i}(-1)$. If $p_a(\tilde C_i)=1$, then  $h^0(\tilde\kl_i(p_i))=1$ and $h^1(\tilde\kl_i(p_i))=0$, so that ${\phi_i}_*\tilde\kl_i(p_i)=\ko_{Q_i}\oplus \ko_{Q_i}(-1)$ also in this case. 
So for $i=1,2$ we choose the linearization of $\tilde\kl_i(p_i)$ in such a way that $\ko_{Q_i}$ is the invariant summand; this is the same as choosing a linearisation of $\kl \otimes \omega_C$ as claimed in $(i)$. Note that choosing a linearisation of  $\kl\otimes \omega_C$ is the same as choosing  a linearisation of $\kl$, because $\omega_C$ is invertible and has a chosen linearisation (cf. Remark \ref{rem: linearisation} and Lemma \ref{lem:  linearisation B}). Note also that if $H^0(\kl)$ contains a nonzero section $x$, then $x^3\in H^0(\kl\otimes \omega_C)$ is invariant, hence so is $x$ and claim (i) is proven.
Statement $(ii)$ for odd $n>0$ follows from the above remarks. 

We consider now  ${\varphi_i}_*\tilde\kl_i$: we have  $a_i=h^0(\tilde \kl_i)=h^1(\tilde\kl_i)$, so we have   \[{\varphi_i}_*\tilde\kl_i=\ko_{Q_i}(-1+a_i) \oplus \ko_{Q_i}(-1-a_i).\] If $a_i=0$, the two eigensheaves are isomorphic and claim $(ii)$ for even $n$ follows immediately.
If $a_i=1$, by $(i)$ we know that the invariant subsheaf is  the  first one and we get $(ii)$ also in this case. 
\end{proof}

We are now going to prove  cases $B(0),B(1)$ and $B(2)$ of Theorem  \ref{thm: all half canonical rings}.

\subsubsection{Type $B(1)$} \label{sect: B1}
Let $( \kl, \mu)$  be a ggs curve of type $B(1)$ so that $(a_1,a_2)=(0,1)$. Here we prove that $R(C, \{\kl, \omega_C\})$ is isomorphic to  $\IC[x, y, w, v, z, u]/I$
where the degrees are $\deg(x, y, w, v, z, u)=(1,2, 3,4,5,6)$ and  $I$ is generated by the following relations:
\[2\times 2\text{ minors of }\begin{pmatrix}0&y&w&z\\x&w&v&u\end{pmatrix}\  
\text{ and }
\begin{array}{rcl}
z^2 & = & yg_8(y,v) \\
zu & = & wg_8(y,v) \\
u^2 & = & vg_8(y,v) + x^4h_8(x,v)
\end{array}
\]
With respect to the  linearisation of Lemma \ref{lem: B decomposition} the involution $\iota$ acts on $R(C, \{\kl, \omega_C\})$ by $(x,y,w,v,z,u)\mapsto (x,y,w,v,-z,-u)$.
\bigskip

Let us denote the non-zero section  of $\kl$ by $x$.  By Lemma \ref{lem: B decomposition} the section $x$ is invariant  and supported on $C_2$. 

Since $C$ is reduced we have $x^2\neq 0$ so that by Proposition \ref{prop: types} and its proof we can write 
\begin{equation}\label{eq: can ring B1} 
R(C, K_C) \isom \IC[x^2, y,v,  u]/ ( 
x^2y, u^2 -f_{12}(x^2, y, v)) \supset R(C,K_C)^+ = \IC[x^2, y, v]/(x^2y). 
\end{equation}
Note that $xy=0$,  because $x$ and $y$ are supported on different components of $C$. 

Now we have to add the remaining elements of odd degree. By Lemma \ref{lem: B decomposition} $R_3$ is invariant, generated by $x^3$ and by an element $w$ supported on $C_1$.  Again, since $x$ and $w$ are supported on different components of $C$, we get $xw = 0$. 

Note that $R_4 = R_4^+ = \langle x^4, y^2, v\rangle$ and that $v$ is non-zero at the generic point of both components of $C$, in particular $v$ is not a zero-divisor in $R$. 

Let us now look in degree $5$, where we have the decompositions
\[R_5^+ \isom H^0({\ko_{Q_1}})\oplus H^0({\ko_{Q_2}}(1)) \text{ and } R_5^- \isom  H^0(\ko_{Q_1})\oplus H^0(\ko_{Q_2}(-1)).\]
The first summand of $R_5^+$ is $\langle yw\rangle$, the second is $\langle x^5 , xv\rangle$ and we have a new anti-invariant generator $z$ of degree $5$ supported on $C_1$, which spans $R_5^-$.

We need the following Lemma:
\begin{lem}\label{lem: B2 xs=ys =0}
  If $s, t \in R$ and $yt\neq 0$ then $s = 0$ if and only if $xs = ts =0$. 
 \end{lem}
\begin{proof}
 We are dealing with sections of torsion-free sheaves, so it is enough to check this at the generic point of each component. But $C$ has two (reduced) components and either $x$ or $y$ is non-zero at the generic point of each of them. 
\end{proof}

{\em  End of  the proof of  Theorem  \ref{thm: all half canonical rings}: Type $B(1)$}.

 We have analysed the situation up to degree $5$ already, so we start in degree 
$6$. 
By \eqref{eq: can ring B1}, we have 
\[ R_6^+ = \langle x^2, y, v\rangle \text{ and } R_6^- = \langle u\rangle\]
   and  $u$ is an anti-invariant section that  does not vanish at 
the generic points of $C_1$ and $C_2$. Clearly  
 $w^2$ is an invariant section which is supported on $C_1$,   so we get a relation 
 \begin{equation}\label{eq: rel w squared}
   w^2= \alpha_1 y^3 + \alpha_2 yv
 \end{equation}
Moreover, $x$ and $z$ are supported on different components of $C$ so $xz=0$. We have accounted for all elements in $R_6$. 

Using Lemma \ref{lem: B decomposition}, we claim that $R_7$ decomposes as follows
 \[ R_7^+ = \langle wy^2, wv\rangle \oplus \langle x^7, x^3v \rangle \text{ and } 
R_7^-=\langle yz \rangle \oplus \langle xu\rangle,\]
 thus there are no new relations in degree $7$.  Indeed, 
 consider the first  summand of $R_7^+$ and assume that there is a linear relation $\alpha wy^2+\beta wv=0$. At the generic point of $C_1$ we can divide by $w$ to get $\alpha y^2+\beta v=0$ on $C_1$, namely $\alpha y^2+\beta v=\gamma x^4$ for some scalar $\gamma$. Since $y^2, x^4,v$ are a basis of $R_4$ we get $\alpha=\beta=\gamma=0$ and $wy^2,wv$ are independent.
We can argue in a similar way for  the second summand of $R_7$,  dividing  by $x^3$. The  two summands of $R_7^-$ are easily treated by noting that $xu\neq0$ and $yz\neq 0$.  In this way we have proved the independence of the generators and we can conclude by a dimension count. 
 
 We derive the following decomposition in degree $8$ from the canonical ring:
 \[ R_8 ^+ =  \langle x^8, x^4 v, v^2, y^4, y^2 v \rangle \text{ and } R_8^- = \langle x^2 u , y u \rangle. \]
 The element  $wz\in R_8^-$ is non-zero since both $w$ and $z$ are non-zero at the generic point of $C_1$,   hence up to rescaling $w$ we can write 
  \begin{equation}\label{eq: rel wz} 
 wz = yu.
   \end{equation}

We skip over degree $9$ for now. Again, from our knowledge of the canonical ring, we know that there are no new generators in degree $10$. The only element not accounted for so far is $z^2$, which is non-zero, invariant, and supported on $C_1$,   thus
 \begin{equation}\label{eq: rel z squared}
 z^2 = \mu_1 y^5+ \mu_2y^3v+\mu_3 yv^2.
   \end{equation}

We now analyse the interaction between the relations found so far in order to clean up some coefficients. 
Taking the square of \eqref{eq: rel wz} and substituting \eqref{eq: rel w squared}, \eqref{eq: rel z squared} and the relation $u^2=f_{12}$ in the canonical ring we get
\begin{equation}\label{eq: u z relation}
y^2f_{12}  = (yu)^2 = w^2z^2 = (\alpha_1 y^3 + \alpha_2 yv) (  \mu_1 y^5+ \mu_2y^3v+\mu_3 yv^2).
\end{equation}

Since $f_{12}$ contains the monomial $v^3$ with non-zero coefficient by Proposition \ref{prop: types}, we have $\alpha_2\mu_3\ne 0$.
 With the coordinate transformation $(\alpha_1 y^2 + \alpha_2 v)\mapsto v$
 the relation \eqref{eq: rel w squared} 
becomes 
 \begin{equation}\label{eq: rel w squared new}
   w^2= yv.
 \end{equation}
Substituting back into \eqref{eq: u z relation} we get
\[y^2f_{12}  =  yv (  \mu_1 y^5+ \mu_2y^3v+\mu_3 yv^2) = y^2 v g_8 (y,v)
\]
and therefore 
\[f_{12} = 
vg_8(y,v) + x^4h_8(x,v), \quad z^2 = yg_8(y,v).\]
Note that $y^6$ has been eliminated from the equation by our choice of $v$. 

Now multiplying \eqref{eq: rel wz} by $z$ and substituting \eqref{eq: rel z squared}, we get
\[y(uz) = w(z^2)=ywg_8\]
and since $uz$, $w$ and $y$ are supported on $C_1$, we  may divide by $y$ to obtain 
\[uz=wg_8.\]
 
 Note that by Lemma \ref{lem: generated in low degree}, we have now found all generators of $R$ with the possible exception of degree $9$, which we now check as promised. 
We claim that the decomposition of Lemma \ref{lem: B decomposition} is:
 \[  R_9^+= \langle wy^3, wyv\rangle \oplus
 \langle x^9, x^5v, xv^2 \rangle, \quad
 R_9^- = \langle y^2z, vz\rangle \oplus
 \langle x^3 u \rangle.\]
 We proceed just as we did in degree $7$. At the generic point of $C_1$, we can divide by $w$ to show that $wyv$ and $wy^3$ are linearly independent. Hence the first summand of $R_9^+$ is as claimed. For the second summand, we can divide by $x$ at the generic point of $C_2$.
 
 None of the given elements of  $R_9^-$ are zero by Lemma \ref{lem: B2 xs=ys =0}. Moreover, we may divide by $z$ to get linear independence of $y^2z, vz$ at the generic point of $C_1$. Thus $R_9$ decomposes as claimed.  
 
 Now $wu$ is also an element of $R_9$: it is non-zero, anti-invariant and supported on $C_1$, hence there is a final relation $wu=\lambda_1y^2z+\lambda_2vz$. Multiplying by $y$ and substituting \eqref{eq: rel wz} and \eqref{eq: rel w squared new} we get $yvz=ywu=y(\lambda_1y^2z+\lambda_2vz)=\lambda_1y^3z+\lambda_2yvz$ hence $\lambda_1=0$, $\lambda_2=1$ and the relation is
\[vz=wu.\]
 
Clearly, the given relations are contained in the kernel of the map, so it remains to show that there are no others. 
 To do this, we compute the Hilbert series of $\IC[x,y,w,v,z,u]/I$ (for example, using a computer algebra package), and check that it matches that of $R$. 
 
\subsubsection{Type $B(0)$} \label{sect: B0}
Let $( \kl, \mu)$  be a ggs  structure of type  $B(0)$,  so that $(a_1,a_2)=(0,0)$.
We will prove that $R\isom\IC[y_1,y_2,z_1,z_2,v,t_1,t_2,u]/I$ where the degrees are $\deg(y_1,y_2,z_1,z_2,v,t_1,t_2,u)=(2,2,3,3,4,5,5,6)$  and $I$ is generated by the $2\times2$ minors of the symmetric $4\times 4$ matrix
\[\begin{pmatrix}
y_1 & 0 & z_1 & t_1 \\
0 & y_2 & z_2 & t_2 \\
z_1 & z_2 & v & u \\
t_1 & t_2 & u & g_8
\end{pmatrix}\]
With respect to the  linearisation of Lemma \ref{lem: B decomposition} the involution $\iota$ acts on $R(C, \{\kl, \omega_C\})$ by $(y_1,y_2,z_1,z_2,v,t_1,t_2,u)\mapsto (y_1,y_2,z_1,z_2,v,-t_1,-t_2,-u)$.

Recall from Proposition~\ref{prop: types} and Corollary~\ref{cor: bicanonical map} that the canonical ring of $C$ and its invariant subring are
\[R(C, K_C) \isom \IC[y_1,y_2,v,u]/ ( y_1y_2, u^2 -f_{12}(y_1, y_2, v)) \supset R(C,K_C)^+ = \IC[y_1, y_2, v]/(y_1y_2).\]

We also need the following Lemma, whose proof is similar to that of Lemma \ref{lem: B2 xs=ys =0}.
\begin{lem}\label{lem: B2 y1s=y2s =0}
  If $s, t \in R$ and $y_1t\neq 0$ then $s = 0$ if and only if $y_2s = ts =0$. 
 \end{lem}

\noindent{\em  End of the  proof of  Theorem  \ref{thm: all half canonical rings}: Type $B(0)$}.

Clearly $R_1=0$ and $R_2=R_2^+=\langle y_1,y_2 \rangle$.  The decomposition 
for $R_3$ given in Lemma \ref{lem: B decomposition} yields 
\[R_3^+ \isom H^0({\ko_{Q_1}})\oplus H^0(\ko_{Q_2}) \text{ and }  R_3^-\isom H^0({\ko_{Q_1}}(-1))\oplus H^0(\ko_{Q_2}(-1)),\]
so we choose generators $z_1$, $z_2$ for $R_3=R_3^+$. These are supported on different components of $C$, so $z_1z_2 = 0$. Similarly we get relations $y_1z_2=y_2z_1=0$.

Note that $R_4 = R_4^+ = \langle y_1^2, y_2^2, v\rangle$ and that $v$ is non-zero at the generic point of both components of $C$, in particular $v$ is not a zero-divisor in $R$. 

In degree $5$, we have 
\
\begin{align*}
R_5^+ & \isom H^0({\ko_{Q_1}})\oplus H^0({\ko_{Q_2}}) = \langle y_1z_1\rangle \oplus \langle y_2z_2\rangle\\
R_5^-&\isom H^0(\ko_{Q_1})\oplus H^0(\ko_{Q_2}) = \langle t_1 \rangle \oplus \langle t_2\rangle
\end{align*}

where the new generators $t_1,t_2$ of degree $5$ are anti-invariant because $y_iz_i$ is invariant for $i=1,2$. The $t_i$ are supported on different components of $C$, so we again have relations $t_1t_2=0$ and $t_iz_j=t_iy_j=0$ for $i\ne j$.

In degree $6$ we know from the canonical ring that 
\[R_6^+=\langle y_1^3,y_2^3,y_1v,y_2v\rangle \text{ and } R_6^-=\langle u\rangle.\]
Since $z_1^2$, $z_2^2$ are in $R_6^+$ and supported on $C_1$ respectively $C_2$, we have two relations
\begin{equation}\label{eq: rel zi squared}
z_1^2 = \alpha_1 y_1^3 + \alpha_2y_1v, \quad\quad
z_2^2 = \beta_1 y_2^3 + \beta_2y_2v.
\end{equation}

In the usual way, $R_7$ decomposes as
\begin{align*}
R_7^+ &\isom H^0(\ko_{Q_1}(1))\oplus H^0(\ko_{Q_2}(1))
=\langle y_1^2z_1, z_1v \rangle\oplus\langle y_2^2z_2, z_2v \rangle \\
R_7^- &\isom H^0(\ko_{Q_1})\oplus H^0(\ko_{Q_2}) =  \langle y_1t_1 \rangle \oplus \langle y_2t_2 \rangle
\end{align*}
Indeed $y_it_i\ne0$ and $y_1^2z_1, z_1v$ are linearly independent because we can divide by $z_1$ at the generic point of $C_1$ to get $y_1^2$, $v$ which are linearly independent on $C_1$.

From the involution on the canonical ring, we have
\[R_8^+=\langle y_1^4, y_2^4, y_1^2v, y_2^2v, v^2\rangle \ \text{ and }\ 
R_8^-=\langle y_1u, y_2u\rangle.\]
Now $z_it_i$ are non-zero elements of $R_8$ which are respectively supported on $C_i$ for $i=1,2$. Since $z_it_i$ is anti-invariant for $i=1,2$, we have relations
\begin{equation}\label{eq: rel ziti}
z_1t_1=y_1u, \quad\quad z_2t_2=y_2u
\end{equation}
where we have rescaled $y_1,y_2$ to tidy up the coefficients.

Next we consider degree $9$ where we have the decomposition
 \begin{align*}
R_9^+ &\isom H^0(\ko_{Q_1}(1)) \oplus H^0(\ko_{Q_2}(1))= \langle y_1^3z_1,y_1z_1v \rangle \oplus \langle y_2^3z_2,y_2z_2v\rangle,\\
R_9^- &\isom H^0(\ko_{Q_1}(1))\oplus H^0(\ko_{Q_2}(1))  = \langle y_1^2t_1,vt_1 \rangle \oplus \langle y_2^2t_2,vt_2\rangle.
\end{align*}
As usual, we can show that the listed monomials are linearly independent by dividing by some common factor at the generic point of $C_i$. 
Now for $i=1,2$ we know that $z_iu$ in $R_9$ is non-zero, anti-invariant and is supported on $C_i$. Hence $z_iu$ is in the summand $\langle y_i^2t_i,vt_i \rangle$. Thus we have relations
\begin{equation}\label{eq: rel ziu}
z_1u = \gamma_1y_1^2t_1 + \gamma_2vt_1, \quad\quad 
z_2u = \delta_1y_2^2t_2 + \delta_2vt_2
\end{equation}
Since there are no new generators in degree $9$, it follows from Lemma \ref{lem: generated in low degree}, that $R$ is generated in degrees $\le6$.

We have $t_1^2,t_2^2$ in $R_{10}^+$ each supported on $C_i$ for $i=1,2$. Thus from the canonical ring we get relations
\begin{equation}\label{eq: rel ti sq}
t_1^2 = \lambda_1y_1^5 + \lambda_2y_1^3v + \lambda_3y_1v^2, \quad\quad 
t_2^2 = \mu_1y_2^5 + \mu_2y_2^3v + \mu_3y_2v^2
\end{equation}

Now we study the interactions between relations and clean up several coefficients. Taking the squares of \eqref{eq: rel ziti} and substituting \eqref{eq: rel zi squared}, \eqref{eq: rel ti sq} and the relation $u^2=f_{12}$ from the canonical ring, we get the relations
\begin{equation}\label{eq: rel ui zi}
\begin{split}
y_1^2f_{12} &= y_1^2u^2 = z_1^2t_1^2 = (\alpha_1 y_1^3 + \alpha_2y_1v)
(\lambda_1y_1^5 + \lambda_2y_1^3v + \lambda_3y_1v^2) \\
y_2^2f_{12} &= y_2^2u^2 = z_2^2t_2^2 = (\beta_1 y_2^3 + \beta_2y_2v)
(\mu_1y_2^5+\mu_2y_2^3v+\mu_3y_2v^2)
\end{split}
\end{equation}
Since $f_{12}$ contains the monomial $v^3$ with non-zero coefficient it follows that $\alpha_2$ and $\beta_2$ must be non-zero. By a coordinate change of the form $v\mapsto v+\frac{\alpha_1}{\alpha_2}y_1^2+\frac{\beta_1}{\beta_2}y_2^2$, noting that $y_1y_2=0$ and by rescaling $z_1,z_2$, the equations \eqref{eq: rel zi squared} become 
\begin{equation}\label{eq: rel zi squared new}
z_1^2 = y_1v, \quad\quad z_2^2 = y_2v.
\end{equation}
Substituting back into \eqref{eq: rel ui zi} we get
\begin{equation}
\begin{split}
y_1^2f_{12} &= y_1v(\lambda_1y_1^5 + \lambda_2y_1^3v + \lambda_3y_1v^2)=y_1^2v(\lambda_1y_1^4 + \lambda_2y_1^2v + \lambda_3v^2) \\
y_2^2f_{12} &= y_2v(\mu_1y_2^5+\mu_2y_2^3v+\mu_3y_2v^2)= y_2^2v(\mu_1y_2^4+\mu_2y_2^2v+\mu_3v^2)
\end{split}
\end{equation}
Therefore
\[f_{12}(y_1,y_2,v)=vg_8(y_1,y_2,v),\quad\quad u^2=vg_8.\]

It follows immediately from the above definition of $g_8$ and the relation $y_1y_2=0$, that the equations
\eqref{eq: rel ti sq} become
\begin{equation}\label{eq: rel ti sq new}
t_1^2 = y_1g_8, \quad\quad 
t_2^2 = y_2g_8.
\end{equation}
Multiplying \eqref{eq: rel ti sq new} by $u$
and substituting \eqref{eq: rel ziti} gives
\begin{equation}
t_1^2u = y_1ug_8 = z_1t_1g_8, \quad\quad t_2^2u = y_2ug_8 = z_2t_2g_8
\end{equation}
and dividing by $t_i$ at the generic point of $C_j$ yields relations
\begin{equation}\label{eq: rel ti u}
t_1u = z_1g_8, \quad\quad t_2u  = z_2g_8.
\end{equation}
 
Multiplying \eqref{eq: rel ziu} by $y_i$ and substituting \eqref{eq: rel ziti} and \eqref{eq: rel zi squared new} we get
\begin{equation}
\begin{split}
y_1vt_1 &= z_1^2t_1 = y_1z_1u = \gamma_1y_1^3t_1 + \gamma_2y_1vt_1\\
y_2vt_2 &= z_2^2t_2 = y_2z_2u = \delta_1y_2^3t_2 + \delta_2y_2vt_2
\end{split}
\end{equation}
hence $\gamma_1=\delta_1=0$ and $\gamma_2=\delta_2=1$ and the relations are
\[z_1u=vt_1,\quad\quad z_2u=vt_2.\]

To show that there are no further relations, we computed the Hilbert series of 
$\IC[y_1,y_2,z_1,z_2,v,t_1,t_2,u]/(\text{relations})$ and checked that it matches with that of $R$.

\subsubsection{Type $B(2)$} \label{sect: B2}

Let $( \kl, \mu)$  be a ggs  structure of type  $B(2)$. We show that
$R \isom \IC[x_1,x_2,v,u]/
\left(x_1x_2 = 0, u^2 = f_{12}(x_1,x_2,v)\right)$. 
With respect to the  linearisation of Lemma \ref{lem: B decomposition} the involution $\iota$ acts on $R(C, \{\kl, \omega_C\})$ by $(x_1,x_2,v,u)\mapsto (x_1,x_2,v,-u)$.

\begin{proof}
We start with a basis for $R_1=\langle x_1, x_2 \rangle$, because $h^0(\kl)=2$. Since $C$ has arithmetic genus $2$, we see that $x_1^2,x_1x_2,x_2^2$ are not linearly independent in $R_2$. The curve $C$ is reduced so we may choose coordinates such that $x_1x_2=0$ and $R_2 = R_2^+ = \langle x_1^2, x_2^2 \rangle$ where $x_1^2=y_1$ and $x_2^2=y_2$. We read off the dimension of $R_n$ using the Hilbert series and continue with $R_3 = \langle x_1^3, x_2^3 \rangle$ and $R_4=R_4^+=\langle x_1^4, x_2^4, v \rangle$.

We tabulate the rest of the computation up to degree $9$, using the canonical ring and the decomposition of $R_{2n+1}$ from Lemma~\ref{lem: B decomposition}.
\begin{center}
\begin{tabular}{ccc}
 \toprule
 $d$ & $R_d^+$ & $R_d^-$ \\
 \midrule
$5$ & $\langle x_1^5,x_1v\rangle \oplus \langle x_2^5,x_2v\rangle$  \\
$6$ & $\langle x_1^6,x_1^2v\rangle \oplus \langle x_2^6,x_2^2v\rangle$ &
$\langle u\rangle$ \\
$7$ & $\langle x_1^7,x_1^3v\rangle \oplus \langle x_2^7,x_2^3v\rangle$ 
& $\langle x_1u\rangle \oplus \langle x_2u\rangle$ \\
$8$ & $\langle x_1^8,x_1^4v\rangle \oplus \langle x_2^8,x_2^4v\rangle \oplus \langle v^2\rangle$ &
$\langle x_1^2u\rangle\oplus \langle x_2^2u\rangle$ \\
$9$ & $\langle x_1^9,x_1^5v,x_1v^2\rangle \oplus \langle x_2^9,x_2^5v,x_2v^2\rangle$ &
$\langle x_1^3u\rangle \oplus \langle x_2^3u\rangle$\\
\bottomrule
\end{tabular}
\end{center}
 The direct summands group together basis elements that are supported on either one of the curves respectively both of them.
For each $R_{2m+1}$, the listed elements form a basis because we can divide by $x_i$ at the generic point of the component $C_i$ to get the basis of $R_{2m}$, which is a summand of the canonical ring. By Lemma \ref{lem: generated in low degree}, we have found all generators of $R$.

The relation $u^2=f_{12}(x_1,x_2,v)$ is induced by the relation in the canonical ring of type B (cf.~Proposition~\ref{prop: types} and Corollary~\ref{cor: bicanonical map}). 

Clearly, the Hilbert series of $\IC[x_1,x_2,v,u]/
\left(x_1x_2 = 0, u^2 = f_{12}(x_1,x_2,v)\right)$ matches that of $R$, so the proof is finished.
\end{proof}

\subsection
{The ``converse'' statement of Theorem  \ref{thm: all half canonical rings}}

In this section we are going to show that each ideal given in Table \ref{tab: half-canonical rings} defines a reduced Gorenstein curve $C$  of genus two such that 
$\omega_C = \ko_{\IP}(2)\restr{C}$  is invertible (and ample)
 and $\kl = \ko_{\IP}(1)\restr{C}\refl{1}$ is a ggs structure on $C$ satisfying the  given conditions on  $h^0$. 
 
 \subsubsection{$A(0)$}
 Let $ \IP= \IP(2,2,3,3)$ be the weighted projective space with variables and degrees $\mat{y_1 & y_2 & z_1 & z_2\\ 2 & 2& 3& 3}$ and 
 let $C\subset \IP(2,2,3,3)$ defined by the ideal    $$I=(z_1^2 - f_{6}(y_1, y_2),z_2^2-g_6(y_1, y_2))  \ \ \mbox{ with }  \ f_6 \neq 0, \ g_6 \neq 0 . $$
 Computing the Hilbert series of  $\IC[y_1,y_2, z_1, z_2]/ I$ we see that $C$ is a curve of genus two. 
 Moreover  we have a $2:1$ map $\psi \colon C\to \IP^1=\IP(2,2)$  given by $(y_1, y_2)$ and for a general  point $P\in \IP^1$, $\psi^{-1}(P)$ consists of two  different smooth points of $C$. 
 Therefore $C$ is reduced. 

 The set of points in $\pp$  where    $\OO_{\pp}(2)$ is not invertible is $\{y_1=y_2=0\}$ and 
  by our assumptions it does not intersect $C$.  By the very definition of the ideal $I$, which is generated by two polynomials of degree 6 having a relation in degree 12,   we have  the resolution
\[ 0 \to \ko_\IP(-12) \to \ko_\IP(-6)^{\oplus 2} \to \ko_\IP\to \ko_C \to 0.  \]
By \cite{Fletcher} we have $\omega_\IP = \ko_{\IP} (-2-2-3-3) = \ko_\IP(-10)$ and by Serre duality the subvariety $C$ of codimension two has canonical sheaf
\[\omega_C\cong \shext^2( \omega_\IP, \ko_C) =\shext^2( \ko_\IP(-10), \ko_C). \]
The above resolution computes this as  $\omega_C\isom\OO_{\IP}(2)\restr{C}$, which 
  is invertible at all points of $C$, namely $C$ is Gorenstein.  
Finally,  $h^0(\kl)=h^0(\ko_{\IP}(1))=0$. 

\subsubsection{$A(1)$}
 Let $\IP=  \IP(1,2,5)$ be the weighted projective space with variables and degrees $\mat{x & y & z\\ 1 & 2 & 5}$,
and let $C\subset \IP(1,2,5)$ defined by the ideal    $I=(z^2 - f_{10}(x,y))$, with $f_{10}\neq 0$.

Arguing as in the previous case,  $C$ is a reduced curve of genus two and we have   $\omega_C\cong \OO_{\IP}(2)\restr{C}$ is invertible and  $h^0(\kl)=h^0(\ko_{\IP}(1))=1$. 

\subsubsection{$B(0)$} 
 Let $\IP= \IP(2,2,3,3,4,5,5,6)$ be the weighted projective space with variables and degrees  \[\mat{y_1 & y_2 & z_1 & z_2 & v & t_1 & t_2 & u\\ 2 & 2& 3& 3 & 4 &  5& 5 & 6}\]
and let $C$ be  defined by the ideal    
$$I=\rk
  \begin{pmatrix}
y_1 & 0 & z_1 & t_1 \\
0 & y_2 & z_2 & t_2 \\
z_1 & z_2 & v & u \\
t_1 & t_2 & u & g_8
\end{pmatrix}
\leq 1  \ \ \mbox{ where } \ g_8 =\alpha v^2 +  v h_4(y_1,y_2) + k_8(y_1,y_2)  \ \mbox{ with } \alpha \neq 0 $$
We can write  $C=C_1\cup C_2$ where $C_i$ is the restriction of $C$ to $y_i=0$.  

\noindent Consider $C_1$: at points where $y_2\ne 0$  its ideal is generated by  $t_1,z_1$ and by the 2-by-2 minors of  the matrix
$
  \begin{pmatrix}
 y_2 & z_2 & t_2 \\
 z_2 & v & u \\
 t_2 & u & g_8
\end{pmatrix}
$. 

Computing  the resolution of this ideal (e.g., taking the Eagon--Northcott complex or using  MAGMA or Macaulay2, see the attached file), one sees that our condition on $g_8$ implies that  it is a prime  ideal, and 
 computing the Hilbert series  we see that it defines an irreducible curve $D_1$ with  $p_a(D_1)=1$. In particular $D_1$ is Gorenstein  by \cite[Lemma 1.19]{catanese82}. 
 An analogous computation for $C_2$ shows that $C$ contains $D_1\cup D_2$, where $D_1$ and $D_2$ are irreducible Gorenstein curves with arithmetic genus 1 and $C=D_1\cup D_2$ except possibly at the intersection of $C$ with $\{y_1=y_2=0\}$. This intersection consists only of the point  $P:= [0,0,0,0,1,0,0,\xi ]$, where $\alpha\xi^3=1$, and 
an explicit  computation shows  that  $C$ has a node at $P$
 since for a generic $g_8$ the two curves are smooth and transverse at $P$.   Hence $C$ is reduced at all points, $C_i=D_i$, $i=1,2$,  and $C=C_1\cup C_2$. 
Summing up, we have shown that $C$  is a  reduced Gorenstein curve of genus two of type $B$.   Hence $\omega_C$ is a line bundle.

Finally, we have $\omega_C\restr{C_i} \cong \OO_{C_i}(P)\cong \OO_{\IP}(2)\restr{C_i}$  (this can be seen for instance considering the hyperplane $(y_j=0)|_{C_i}$ for $i\neq j$).
 
 Note that $C$ is contained in the locus of $\pp$ where $\OO_{\IP}(2)$ is invertible, hence $\OO_{\IP}(2)\restr C$ is a line bundle. 
Since the  two components $C_1$ and $C_2$ of $C$ meet transversally only at one point, the natural map $\Pic(C)\to \Pic(C_1)\times \Pic(C_2)$ is an isomorphism, and therefore $\omega_C\cong \OO_{\IP}(2)$.\footnote{Alternatively, one can use Magma or Macaulay2 to show that $\omega_C=\OO_C(2)$.}
  Finally we have  $h^0(\kl)=h^0(\ko_{\IP}(1))=0$ by construction.

\subsubsection{$B(1)$}
 Let $\IP=  \IP(1,2,3,4,5,6)$ be the weighted projective space with variables and degrees 
 \[\mat{ x&  y & w & v & z & u\\ 1&2& 3& 4& 5& 6},\]
and let $C\subset \IP(1,2,3,4,5,6)$ be defined by the ideal

\[ \rk \begin{pmatrix}0&y&w&z\\x&w&v&u\end{pmatrix}\leq 1,
\begin{array}{rcl}
z^2 & = & yg_8(y,v) \\
zu & = & wg_8(y,v) \\
u^2 & = & vg_8(y,v) + x^4h_8(x,v)
\end{array}  
\]
where  $g_8 = \alpha v^2 + vh_4(y) + k_8(y)$ with $\alpha \neq 0$.

Arguing as in the previous case we define $C_1$, resp. $C_2$, as the restriction of $C$ to $x=0$, resp. $y=0$. Computations similar to the previous ones show   that the $C_i$ are irreducible Gorenstein curves with $p_a(C_i)=1$, they meet 
transversely only at the point $P:=(0:0:0:0:1:0:0:\xi )$, where $\alpha \xi^3=1$,  and that $C$ is reduced and therefore equal to $C_1\cup C_2$. 
So $C$ is a  Gorenstein genus 2 curve of type $B$.  Finally one checks that $\omega_C$ is isomorphic to $\OO_{\pp}(2)\restr C$ and $h^0(\kl)=h^0(\ko_{\IP}(1))=1$  by the same arguments as in case $B(0)$.

\subsubsection{$B(2)$}
 Let $\IP=\IP(1,1,4,6)$ be the weighted projective space with variables and degrees 
 \[\mat{ x_1& x_2 & v & u \\ 1&1& 4& 6},\]
and let $C\subset \IP(1,1,4,6)$ defined by the ideal    $$I=(x_1x_2,\ \  u^2 - f_{12}(x_1, x_2, v))
\ \ \mbox{where $v^3$ appears in $f_{12}$ with non-zero coefficient. }$$

One argues exactly as in cases $B(0)$ and $B(1)$ and proves that $C$ is a reduced  Gorenstein curve of genus 2 that can be written $C=C_1\cup C_2$, where  $C_i=C\cap \{x_i=0\}$, $i=1,2$. The curves $C_i$ are irreducible Gorenstein with $p_a(C_i)=1$ and meet transversely at exactly one point. The  line bundle $\omega_C$ is the restriction of $\OO_{\pp}(2)$ and $h^0(\kl)=h^0(\ko_{\IP}(1))=2$.

\hfill\break{\bf Data availability}

 During the work on this publication, no data sets were generated, used or
analyzed. Thus, there is no need for a link to a data repository.

\hfill\break{\bf Declarations}

{\bf Conflict of interest. }  On behalf of all authors, the corresponding author states that there is
no conflict of interest.

 \def\cprime{$'$}


\end{document}